\documentclass[a4paper,12pt]{amsart}

\pdfoutput=1

\usepackage{amsmath,amssymb,amsthm,url}
\usepackage[utf8]{inputenc}
\usepackage[T1]{fontenc}
\usepackage{libertine}
\usepackage[libertine,cmintegrals,cmbraces]{newtxmath}
\usepackage{bbold}
\usepackage[shortlabels]{enumitem}
\usepackage{mathtools}
\usepackage{adjustbox}
\usepackage{microtype}
\usepackage{xcolor}
\usepackage{tikz}
\usepackage{tikz-cd}
\usepackage{nicefrac}
\usepackage{dsfont}
\usepackage{todonotes}
\usepackage{a4wide}
\usepackage{quiver}
\usepackage{anyfontsize}

\newcommand*\circled[1]{\tikz[baseline=(char.base)]{
            \node[shape=circle,draw,inner sep=2pt] (char) {#1};}}

\AtBeginDocument{%
   \def\MR#1{}
} 
\usepackage{euscript}
\usepackage[all]{xy}

\usepackage[pagebackref]{hyperref}
  
\hypersetup{%
  bookmarksnumbered=true,
  colorlinks=true,%
  linkcolor=blue,%
  citecolor=blue,%
  filecolor=blue,%
  menucolor=blue,%
  urlcolor=blue,%
  pdfnewwindow=true,%
  pdfstartview=FitBH}

\usepackage{xcolor}
\definecolor{seagreen}{RGB}{46,139,87}
\definecolor{maroon}{RGB}{128,0,0}
\definecolor{darkviolet}{RGB}{148,0,211}
\definecolor{twelve}{RGB}{100,100,170}
\definecolor{thirteen}{RGB}{100,150,50}
\definecolor{fourteen}{RGB}{200,0,0}
\definecolor{fifteen}{RGB}{0,200,0}
\definecolor{sixteen}{RGB}{0,0,200}
\definecolor{seventeen}{RGB}{200,0,200}
\definecolor{eighteen}{RGB}{0,200,200}

\newcommand{\mc}[1]{\mathcal{#1}}
\newcommand{\bb}[1]{\mathbb{#1}}

\newcommand{\mbf}[1]{\mathbf{#1}}

\newcommand{\es}[1]{\EuScript{#1}}
\renewcommand{\sf}[1]{{\mathsf{#1}}}

\DeclareMathOperator{\hocolim}{\mathsf{hocolim}}
\DeclareMathOperator{\holim}{\mathsf{holim}}
\DeclareMathOperator{\hofibre}{\mathsf{hofib}}

\newcommand{\s}{{\sf{Spec}}}
\DeclareMathOperator{\T}{\mathsf{Top}_\ast}

\newcommand{\poly}[1]{\mathsf{Poly}^{\leq #1}}
\newcommand{\homog}[1]{\mathsf{Homog}^{#1}}

\newcommand{\Fun}{\sf{Fun}}
\DeclareMathOperator{\Hom}{\mathsf{Hom}}

\DeclareMathOperator{\nat}{\mathsf{nat}}

\DeclareMathOperator{\id}{\mathsf{Id}}
\DeclareMathOperator{\ind}{\mathsf{ind}}
\DeclareMathOperator{\res}{\mathsf{res}}

\newcommand{\R}{\bb{R}}

\newcommand{\op}{\mathsf{op}}

  \newcommand{\adjunction}[4]{
\xymatrix{
#1:#2 \ar@<.5ex>[r] &
\ar@<.5ex>[l] #3:#4
}}

\newtheorem{thm}{Theorem}[subsection]
\newtheorem{prop}[thm]{Proposition}

\newtheorem{lem}[thm]{Lemma}
\newtheorem{cor}[thm]{Corollary}

\newtheorem{xxthm}{Theorem}

\newtheorem{xxconj}{Conjecture}

\theoremstyle{definition}
\newtheorem{definition}[thm]{Definition}
\newtheorem{ex}[thm]{Example}
\newtheorem{exs}[thm]{Examples}
\newtheorem{rem}[thm]{Remark}
\newtheorem{warning}[thm]{Warning}

\begin{document}
\title{Koszul duality and a classification of stable Weiss towers}

\author{Connor Malin}
\address[Malin]{Max Planck Institute for Mathematics, Vivatsgasse 7, 53111 Bonn}
\email{malin@mpim-bonn.mpg.de}
\author{Niall Taggart}
\address[Taggart]{IMAPP, Radboud University Nijmegen, The Netherlands}
\email{niall.taggart@ru.nl}

\begin{abstract}
We introduce a version of Koszul duality for categories, which extends the Koszul duality of operads and right modules. We demonstrate that the derivatives which appear in Weiss calculus (with values in spectra) form a right module over the Koszul dual of the category of vector spaces and orthogonal surjections, resolving conjectures of Arone--Ching and Espic. Using categorical Fourier transforms, we then classify Weiss towers. In particular, we describe the $n$-th polynomial approximation as a pullback of the $(n-1)$-st polynomial approximation along a ``generalized norm map''.
\end{abstract}
\maketitle

\setcounter{tocdepth}{1}
{\hypersetup{linkcolor=black} \tableofcontents}

\section{Introduction}
\subsection{Background}
Functor calculus comprises numerous techniques suitable to study functors $F: \mc{C} \to \mc{D}$ between $\infty$-categories by employing a sequence of ``polynomial'' approximations
\[\begin{tikzcd}
	&& F \\
	{P_\infty F} & \cdots & {P_nF} & \cdots & {P_1F} & {P_0F}
	\arrow[curve={height=6pt}, from=1-3, to=2-1]
	\arrow[from=1-3, to=2-3]
	\arrow[curve={height=-6pt}, from=1-3, to=2-5]
	\arrow[curve={height=-6pt}, from=1-3, to=2-6]
	\arrow[from=2-1, to=2-2]
	\arrow[from=2-2, to=2-3]
	\arrow[from=2-3, to=2-4]
	\arrow[from=2-4, to=2-5]
	\arrow[from=2-5, to=2-6]
\end{tikzcd}\]
with the property that the difference between successive approximations is ``computable''. The original version of a calculus of functors was introduced by Goodwillie in the seminal papers~\cite{GoodCalcI, GoodCalcII, GoodCalcIII} in order to study functors between the categories of spaces and spectra. This functor calculus, which we refer to as Goodwillie calculus, remains the most widely studied and applied version of functor calculus, and its applications are seen throughout homotopy theory, see e.g.,~\cite{DundasGoodwillieMcCarthy, Heutsvn, BehrensRezk}. In this setting, Goodwillie demonstrated that the difference between $P_nF$ and $P_{n -1}F$ is completely classified by the ``derivative spectrum'' $\partial_n(F)$, which is a spectrum with an action of $\Sigma_n$, the symmetric group on $n$ letters.

In the landmark papers~\cite{ACOperads, ACClassification}, Arone and Ching proved incredible structural properties of Goodwillie calculus. Together, they expanded Ching's~\cite{ChingOperad} observation that the sequence $\partial_\ast(\id) =\{\partial_n(\id)\}_{n \in \bb{N}}$ of derivatives of the identity functor on pointed spaces form an operad, by demonstrating that the derivatives of various classes of functors inherit the structure of an operadic (bi)module over the derivatives of the identity. These module structures were the key ingredients in formulating and proving a chain rule for the derivatives of composites of functors and for classifying Goodwillie towers. This operad $\partial_\ast(\mathsf{Id})$ is in fact computable and miraculously coincides with the spectral Lie operad \cite{ChingOperad}. In combination with an increased understanding of the Lie operad, these results were used to advance our understanding of the unstable homotopy groups of spheres~\cite{BehrensEHP, BoydeGrowth}.

In this paper we investigate analogous structural properties of Weiss calculus~\cite{WeissOrthogonal}, a geometric cousin of Goodwillie calculus which applies to functors from the category $\mathsf{Vect}_\bb{R}$ of Euclidean spaces to spaces or spectra. Weiss calculus both applies to and relies on the geometry of Euclidean spaces and is a powerful homotopical tool which can be used to study problems of geometric origin, see e.g.,~\cite{ALTV, KrannichRandal-Williams, Hu, Munoz-Echaniz, CHO}.

\subsection{Main Results}

The goal of this paper is to establish the foundations of the Arone--Ching program in the setting of stable Weiss calculus, i.e., for functors from the category of Euclidean spaces and linear isometries to the category of spectra.  In this work we  only explicitly deal with the real version of Weiss calculus, sometimes called orthogonal calculus, though our results extend to unitary calculus~\cite{TaggartUnitary} in the expected ways. There are numerous other variants of Weiss calculus which one could expect to apply these techniques to ~\cite{Arro, TaggartReality, TaggartLocalizations, CarrTaggart}.

\subsection*{The algebraic framework} In this paper, we work with spectrally enriched categories $\mathcal{C}$ with an augmentation to a fixed spectrally enriched category $A_\ast$. We are interested in the category $\mathsf{RMod}_\mc{C}$ of right $\mathcal{C}$-modules,  defined as the spectral presheaves on $\mathcal{C}$. We construct an $A_\ast^\mathsf{op}$-augmented category $K(\mathcal{C})$, which we call the \emph{Koszul dual of $\mc{C}$} and extend Koszul duality to right modules over $\mc{C}$. 

\begin{xxthm}[Theorem~\ref{thm:koszul duality for right modules}]
Koszul duality forms a spectrally enriched Quillen adjunction
\[
\adjunction{K}{\mathsf{RMod}_{\mathcal{C}}}{\mathsf{RMod}_{K(\mathcal{C})}^{\mathsf{op}}}{K^{-1}}.
\]
If $\mc{C}$ is dualizable (Definition \ref{def: dualizable category}), then the derived (co)units are equivalences for level-finite right modules. 
\end{xxthm}

In particular, if $P$ is a reduced, level-finite operad in spectra, applying this to $\mathsf{Env}(P)$, the envelope of $P$, resolves positively a conjecture of Ching \cite[Theorem 9.1]{BehrensRezk} that the category of right $P$-modules is equivalent to the category of right $K(P)$-modules, under (nonequivariant) finiteness conditions. 

Underlying every spectral category is a spectral $\infty$-category, see Section \ref{subsection: spectral homotopy}.

\begin{xxthm}[Theorem \ref{thm: double koszul dual}]
    If $C$ is dualizable, then there is an equivalence of augmented spectrally enriched $\infty$-categories
    \[K(K(C)) \simeq C.\]
\end{xxthm}

\subsection*{Koszul duality in Weiss calculus} Our main application of Koszul duality is to construct a model for the Weiss derivatives of a functor $F: \mathsf{Vect}_\bb{R} \to \s$ which supports the action of the ``orthogonal Lie envelope'' $K(\mathsf{OEpi})$, where, $\mathsf{OEpi}$ is the category of ``orthogonal surjections'' defined as the opposite of the category $\mathsf{Vect}_\bb{R}$.

\begin{xxthm}[Theorem~\ref{thm:koszul derivatives are derivatives}]
The Weiss derivatives have the structure of a right module over $K(\mathsf{OEpi})$. In particular, the  Weiss derivatives assemble into a left Quillen functor
    \[
\partial_\ast: \mathsf{Fun}(\mathsf{Vect}_\bb{R},\s) \longrightarrow \mathsf{RMod}_{K(\mathsf{OEpi})}.
\]
\end{xxthm}

For $\mc{C}$ either the category of orthogonal groups $O(\ast)$ or $K(\mathsf{OEpi})$, we provide an equivalence of stable $\infty$-categories between $n$-polynomial functors and $n$-truncated coalgebras over a comonad induced by the derivatives. 

\begin{xxthm}[Theorem~\ref{thm: derivative comonadic}]
Let $n$ be a non-negative integer. There is an adjunction of stable $\infty$-categories
\[
\adjunction{\partial_\ast}{\poly{n}(\mathsf{Vect}_\bb{R}, \s)}{\sf{RMod}_\mc{C}}{\Phi},
\]
which is comonadic, and so induces an equivalence of $\infty$-categories
\[
\poly{n}(\mathsf{Vect}_\bb{R}, \s) \cong \sf{CoAlg}^{\leq n}_{\partial_\ast\Phi}(\sf{RMod}_{\mc{C}}),
\]
between $n$-polynomial functors and $n$-truncated $\partial_\ast\Phi$-coalgebras in right $\mc{C}$-modules. In particular, For any functor $F: \mathsf{Vect}_\bb{R} \to \s$ the Weiss tower of $F$ is equivalent to the tower 
\[
\sf{cobar}(\Phi, \partial_\ast\Phi, \partial_{\leq \ast}F)
\]
induced by the truncation tower of $\partial_\ast F$ as a right $\mc{C}$-module.
\end{xxthm}

Using this coalgebraic data we produce a version of the Kuhn-McCarthy pullback square which classifies the map $P_nF \to P_{n-1}F$  by means of a universal fibration\footnote{As in Goodwillie calculus, one can also express the right-hand corner as a homotopy orbits ~\cite[Corollary 4.17]{ACClassification}.}.

\begin{xxthm}[Corollary~\ref{cor: McCarthy-Kuhn}]\label{thm: intro pullback}
Let $n$ be a non-negative integer. For every functor $F: \mathsf{Vect}_\bb{R} \to \s$, there is a natural homotopy pullback square
\[\begin{tikzcd}
	{P_nF(V)} & {(D_{O(n)}^\vee \wedge S^{nV}/\sf{DI}_n(V) \wedge \partial_n F)^{hO(n)}} \\
	{P_{n-1}F(V)} & {(\Sigma \sf{DI}_n(V) \wedge \partial_n F)_{hO(n)}}
	\arrow[from=1-1, to=2-1]
	\arrow[from=1-1, to=1-2]
	\arrow[from=1-2, to=2-2]
	\arrow[from=2-1, to=2-2]
\end{tikzcd}\]	
where $\mathsf{DI}_n(V)$ is the linear fat diagonal.
\end{xxthm}

In many ways, this coalgebra structure is reminiscent of a divided power coalgebra, see ~\cite[\S6.1]{HeutsGoodwillieApprox}. Up to homotopy, such data can be packaged into a ``divided power right module'' \cite[Theorem 0.5]{ACClassification}.

\subsection{Similarities and differences with Goodwillie calculus}

There is a direct relationship between Goodwillie calculus and Weiss calculus through the one-point compactification functor $(-)^+ : V \mapsto V^+$, see e.g.,\cite{BarnesEldredComparing, AroneAkFree, BehrensEHP}. In particular, for a functor $F: \T \to \s$, the Weiss tower of the composite $F \circ (-)^+$ agrees with the Goodwillie tower of $F$ evaluated on spheres. Thus, one expects that the stable Weiss calculus version of the Arone--Ching program should be similar to the case of functors $\T \rightarrow \s$. 

Our approach to Weiss calculus is to first verify a conjecture of Behrens \cite{behrensOrthNote}, that the homogeneous functors $S^{n(-)}=\Sigma^\infty(\mathbb{R}^n \otimes (-))^+$ act like characters in the sense of harmonic analysis. One defines Fourier transforms against them, which land in the category of right $\sf{OEpi}$-modules. We then show that the application of another Fourier transform, Koszul duality, takes us to the Weiss derivatives, implementing conjectures of Espic \cite{espic} and solving Arone--Ching's question on how to produce right module structures on derivatives in the absence of a functor $\s \rightarrow \mathsf{Vect}_\mathbb{R}$ \cite[Page 5]{ACOperads}.

The category $\mathsf{OEpi}$ can be thought of as a lift of the category $\mathsf{Surj}=\mathsf{Env}(\mathsf{Comm})$ of finite sets and surjections. As such, $K(\mathsf{OEpi})$ is the inner product space analog of the Lie operad. In future work, we verify that the same strategy implemented in this paper works in Goodwillie calculus, which we will then use to study the interaction of Goodwillie and Weiss calculus.

\subsection{Future work} We conclude the introduction by providing a number of conjectures relating to this work and possible extensions thereof, which we plan to return to in future work.

\subsection*{Products in Koszul duality and Weiss calculus}

Suppose that $\mc{C}$ is augmented over $A_\ast$, and that $\mc{C}$ comes equipped with a symmetric monoidal structure. The category $\mathrm{RMod}_\mc{C}$ inherits a symmetric monoidal structure $\circledast$ given by \textit{Day convolution}. It is natural to ask how this interacts with Koszul duality.

One can see that \emph{if} the functor 
\[
\mathsf{Triv}_\mc{C}: \mathsf{RMod}_{A_\ast} \rightarrow \mathsf{RMod}_\mc{C},
\]
is symmetric monoidal with respect to derived Day convolution, then $K(\mc{C})$ is a symmetric monoidal category, such that the Koszul duality adjunction is symmetric monoidal with respect to Day convolution.

This condition is known to be satisfied when $\mc{C}=\mathrm{Env}(P)$ for an operad $P$, since Day convolution is computed on the underlying symmetric sequences. In general, this symmetric monoidality condition is very restrictive and for dualizable $\mc{C}$ can be seen to imply splittings of the mapping spectra $\mc{C}(n,m)$ in terms of $\mc{C}(n,1)$. In the case of $\mathsf{OEpi}$, splittings of similar natures do exist \cite{MillerSplitting}, and so we conjecture:

\begin{xxconj}
    For $\mathsf{OEpi}$
    \[\mathsf{Triv}: \mathsf{RMod}_{O(\ast)} \rightarrow \mathsf{RMod}_\mathsf{OEpi}\]
    is symmetric monoidal with respect to Day convolution.
\end{xxconj}

The motivation behind this conjecture is that it implies a \textit{product rule} for Weiss calculus.

\begin{xxconj}
For $F, G: \mathsf{Vect}_\bb{R} \to \s$, there is an equivalence of right $K(\mathsf{OEpi})$-modules
    \begin{align*}
        \partial_\ast(F \wedge G) &\simeq \partial_\ast F \circledast \partial_\ast G.
    \end{align*}
 
\end{xxconj}

\subsection*{\textbf{The mixed chain rule}}
It was observed by Arone and Ching~\cite{ACOperads} that orthogonal sequences are left tensored over symmetric sequences with a composition product type formula. For $P$ a symmetric sequence and $Q$ an orthogonal sequence, one may define the composition to be the orthogonal sequence given in level $n$ by
\[
(P \circ Q)(n) = \bigvee_{k =1}^n~\bigvee_{n_1 + \cdots + n_k = n} O(n) \wedge_{O(n_1, \dots, n_k)} P( k) \wedge Q(n_1) \wedge \cdots Q(n_k),
\]
where $n_1 \geq n_2 \geq \cdots n_k > 0$, and $O(n_1, \dots, n_k)$ is the normalizer of $O(n_1) \times  \cdots \times O(n_k)$ in $O(n)$, or equivalently, $O(n_1, \dots, n_k)$ is the semi-direct product of $O(n_1) \times  \cdots \times O(n_k)$ in $O(n)$ with the subgroup of $\Sigma_k$ which permutes the indices $i$ for which the corresponding $n_i$ are equal. When $F(\mathbb{R}^\infty)=\ast$, Arone and Ching conjectured a chain rule for the composite
\[
\mathsf{Vect}_\bb{R} \xrightarrow{\ F\ } \mathsf{Spec} \xrightarrow{\ G \ } \mathsf{Spec}
\]
\[
\partial_\ast(G \circ F) \simeq \partial_\ast G \circ \partial_\ast{F}.
\]
Notably, the left-hand side are Weiss derivatives while the right-hand side uses the above tensoring of the Goodwillie derivatives with the Weiss derivatives. Such a result would fall in line with the chain rule in Goodwillie calculus \cite{ACOperads}. We propose a reformulation of this conjecture in terms of our conjectured Day convolution on $\mathrm{RMod}_{K(\mathsf{OEpi})}$.

For any symmetric monoidal category $\mc{C}$, there is a canonical composition product of symmetric sequences and right modules given by
\[
(-) \odot (-) : \mathsf{SymSeq} \times \mathsf{RMod}_{\mc{C}} \longrightarrow \mathsf{RMod}_{\mc{C}},
\]
\[S \odot R := \int^{n \in \Sigma_\ast} R^{\circledast n} \wedge S(n).\]
 This formula is a generalization of the ``substitution product'' of symmetric sequences, see e.g.,~\cite[\S6.3]{CoendCalc} or \cite[Definition 4.2]{MayZhangZou}. In the case $\mc{C}=\Sigma_\ast$ or $\mc{C}=\mathsf{Env}(P)$ for some operad $P$, this coincides with the usual composition product defined in terms of the symmetric group actions, but for a general $\mc{C}$ uses the entire right module structure of the right-hand side.

\begin{xxconj}
For a composite of functors of the form
\[
\mathsf{Vect}_\bb{R} \xrightarrow{\ F\ } \mathsf{Spec} \xrightarrow{\ G \ } \mathsf{Spec},
\]
such that $F(\bb{R}^\infty)=\ast$, there is an equivalence
\[
\partial_\ast(G \circ F) 
\simeq \partial_\ast G \odot \partial_\ast F
\]
of right $K(\mathsf{OEpi})$-modules.
\end{xxconj}

So far, we have only discussed the stable variant of Weiss calculus. In Goodwillie calculus, the derivatives of a functor $F: \T \to \T$ have the structure of a bimodule over $\partial_\ast(\id)$, the derivatives of the identity on pointed spaces, we expect the natural generalization to Weiss calculus to hold.

\begin{xxconj}
The derivatives of a functor $F: \mathsf{Vect}_\bb{R} \to \T$ have the structure of a $(\partial_\ast(\id) - K(\mathsf{OEpi}))$-bimodule
such that for 
\[
\mathsf{Vect}_\bb{R} \xrightarrow{\ F\ } \T \xrightarrow{\ G \ } \T,
\]
with $F(\bb{R}^\infty)=\ast$, there is an equivalence
\[
\partial_\ast(G \circ F) \simeq \partial_\ast(G) \odot_{\partial_\ast(\id)} \partial_\ast(F),
\]
of $(\partial_\ast(\id) - K(\mathsf{OEpi}))$-bimodules.
\end{xxconj}

Such a chain rule and a well-behaved theory of bimodule Koszul duality would allow one to recover the derivatives of a functor $F: \mathsf{Vect}_\bb{R} \to \T$ through a suitable cobar construction applied to $\partial_\ast (\Sigma^\infty F )$. Of particular note is an application to spaces of embeddings. Arone~\cite{AroneEmbeddings} gave a closed formula for the Weiss derivatives of the functor $\Sigma^\infty \overline{\mathsf{Emb}}(M,N \times (-))$, for manifolds $M$ and $N$, and we expect a computation of the unstable derivatives to be of significant use.

\subsection*{The Koszul dual of orthogonal epimorphisms}

In this paper, we construct a map
\[K(\mathsf{OEpi}) \rightarrow \langle \partial_i \rangle_{i \in \mathbb{N}},\]
where $\langle \partial_i\rangle_{i \in \mathbb{N}}$ denotes the full subcategory of $\mathsf{Fun}(\mathsf{Fun}(\mathsf{Vect}_\bb{R},\s),\s)$ spanned by the derivative functors. We expect that this functor is an equivalence. For ease of narration, the rest of this section is stated modulo smash products with adjoint representations. By Section \ref{subsection:spanier whitehead model}, there is an equivalence
\[\nat(\partial_i,\partial_j) \simeq \partial_i(S^{j(-)}/\sf{DI}_j(-))^\vee.\]
Arone has explained to the authors that the latter can be computed as sections of a certain bundle
\[P_{j-i} \rightarrow \xi(i,j) \rightarrow \mathsf{Gr}(i,j)\]
where $P_{j-i}$ is the Spanier--Whitehead dual of the nerve of the topological poset of subspaces of $\mathbb{R}^{j-i}$ and $\mathsf{Gr}(i,j)$ is the Grassmannian of $i$-planes in $\mathbb{R}^j$. 

\begin{xxconj}[Arone]
The collection of such sections support the structure of a category, and this category is equivalent to $K(\mathsf{OEpi})$. Moreover, $\mathsf{OEpi}$ is a \textit{twisted} Koszul self dual category. More precisely, $K(\mathsf{OEpi})$ may be realized by applying the Thom spectrum functor to a category enriched over stable spherical fibrations where the parametrized mapping spectra lie over the mapping spaces of $\mathsf{OEpi}$.
\end{xxconj}

 In the terminology of \cite[Section 6]{koszulselfduality}, this would imply that the category $\mathsf{OEpi}$ is Poincaré--Koszul. In particular, Koszul duality for $\mathsf{OEpi}$ behaves significantly different from Koszul duality for $\mathsf{Comm}$.

\subsection{Notation}
\begin{itemize}
  \item Given a spectrum $X$, we denote by $X^\vee$ the derived Spanier--Whitehead dual of $X$.
    \item We use $\mbf{L}$ and $\mbf{R}$ to denote left and right derived functors, respectively. 
    \item We will denote the mapping object in an enriched model category $\mc{C}$ by $\mc{C}(-,-)$ and the derived mapping object by $\mc{C}^h(-,-)$.
    \item We use the symbol $\cong$ to refer to isomorphisms and $\simeq$ to refer to weak equivalences.
    \item Given a (pointed) topological category $\mc{C}$, we often abuse notation and write $\mc{C}$ for the spectral category obtained by applying $\Sigma^\infty_+$ ($\Sigma^\infty$) to the mapping spaces. This is justified by the fact that the category of topologically enriched functors $\mc{C} \to \s$ is equivalent to spectrally enriched functor $\mc{C} \to \s$ under this abuse of notation. 
    \item Given a right $\mc{C}$-module $R$, i.e., a functor $R: \mc{C}^\op \to \s$ and a spectral functor $F:\s \rightarrow \s$, we let $F(R)$ denote the composition of functors $F \circ R$.
\end{itemize}

\subsection{Conventions}
\begin{itemize}
    \item We often conflate categories and their preferred skeleta, in particular we index $\mathsf{Vect}_\mathbb{R}$ both by a general finite-dimensional vector space $V$ and by a non-negative integer $n$, corresponding to $\mathbb{R}^n$, since this category plays two distinct roles in our theory.
    \item We say a spectrum is finite if it is equivalent to a finite spectrum. Given a functor $F$, we say it is level-finite if the value on each object is finite. Modules over a ring spectrum are finite if their underlying spectrum is finite. A spectral category is locally finite if all the mapping spectra are finite. In the case we wish to refer to an object as built out of finitely many cells (in the appropriate category), we explicitly say so.
    \item Simplicial enrichments (and simplicial model structures) will always come from applying the singular complex functor to the topological mapping spaces. 
\end{itemize}

\subsection{\textbf{Spectral homotopy theory}}\label{subsection: spectral homotopy}
In this paper, we occasionally alternate between model categories and $\infty$-categories, taking care to verify these results translate as expected. We briefly outline the relationship between these two in the context of spectrally enriched category theory.

Let $\mc{M}$ be a spectral model category, the underlying $\infty$-category $\mc{M}_\infty$ is canonically a spectrally enriched $\infty$-category. One way to see this is to note that any spectral model category is a stable model category, and the $\infty$-categorical localization of a stable model category is a stable $\infty$-category which admits a canonical spectral enrichment. The same is true at the level of adjunctions: in the $\infty$-categorical setting, it is a property to be a spectrally enriched functor and not extra structure, and adjunctions always satisfy this property. Indeed, if 
\[
\adjunction{F}{\mc{M}}{\mc{N}}{G}
\]
is a spectrally enriched Quillen adjunction, then the induced adjunction
\[
\adjunction{F}{\mc{M}_\infty}{\mc{N}_\infty}{G}
\]
is a spectral adjunction of spectral $\infty$-categories. To see this, note that a functor between stable $\infty$-categories preserves finite colimits if and only if it preserves finite limits, and so both functors in an adjunction between stable $\infty$-categories are exact. Under the canonical spectral enrichment, exact functors correspond to spectrally enriched functors for this canonical enrichment, and so any adjunction between stable $\infty$-categories is a spectrally enriched adjunction. For details, see e.g.,~\cite[\S8]{HeineEnrichments}.

We often suspected that the model categorical theory we developed here could equally well have been developed within the framework of spectrally enriched $\infty$-categories. When we embarked on this project, the theory of enriched $\infty$-categories was not developed as far as we would have wanted it to be to implement our theory of Koszul duality. The recent preprint~\cite{HeineBi} appears to develop the missing aspects of enriched $\infty$-category theory that allow for Section \ref{section:koszul duality} to be written at the level of spectrally enriched $\infty$-categories.

\subsection{Acknowledgements}
We would like to thank Gregory Arone, Mark Behrens, Thomas Blom, Michael Ching, and Hadrian Heine for contributing to this paper through extended conversations. Connor Malin started this project as a Ph.D. student at Notre Dame and finished it as a postdoc at the Max Planck Institute for Mathematics in Bonn. Niall Taggart was supported by the European Research council (ERC) through the grant “Chromatic homotopy theory of spaces”, grant no. 950048 and was supported, during the final stages of this project, by the Nederlandse Organisatie voor Wetenschappelijk Onderzoek (Dutch Research Council) Vidi grant no VI.Vidi.203.004.

\section{Enriched presheaves}
\subsection{S-modules}

Let $\s$ denote the symmetric monoidal category of $S$-modules in the sense of~\cite{EKMM}, the objects of which we call \emph{spectra}. This category of spectra is well known to form a symmetric monoidal model category under smash product $\wedge$ for which all objects are fibrant. 
This model of spectra is used for a variety of reasons. It agrees with the established literature of Arone--Ching~\cite{ACOperads,ACClassification} on Goodwillie calculus. It also simplifies many of the symmetric monoidal considerations which arise. This is because the smash product of bifibrant spectra remains bifibrant. The unit $\bb{S}$ of the monoidal structure on $\s$ is not cofibrant, but this does not significantly impact our arguments. For a pointed space $X$, we will denote by $\Sigma^\infty(X) \coloneq \bb{S} \wedge X$ the tensoring of $X$ with the sphere spectrum. 
In a few instances, we instead tensor with a cofibrant replacement of the sphere spectrum $\bb{S}_c$, though the results are weakly equivalent, at least when $X$ is a CW complex.

\subsection{Spectral categories}
In this section, we recall a number of preliminary results on $\s$-enriched categories. For details on enriched category theory, we direct the reader to~\cite{Kelly}.

\begin{definition}
A \emph{spectral category} is a $\mathsf{Spec}$-enriched category. Explicitly, a $\s$-enriched category $\mc{C}$ has a collection of objects and for objects $X,Y,Z\in \mc{C}$, there are mapping spectra $\mc{C}(X, Y)$ and composition maps
\[
\mc{C}(Y, Z) \wedge \mc{C}(X, Y) \longrightarrow \mc{C}(X,Z),
\]
for each $X,Y,Z \in \mc{C}$, which are appropriately associative and unital. Let $\mc{C}$ and $\mc{D}$ be spectral categories. A \emph{spectral functor} is a function $F: \mathsf{ob}(\mc{C}) \to \mathsf{ob}(\mc{D})$ on the objects together with maps 
\[
\mc{C}(X,Y) \longrightarrow \mc{D}(F(X), F(Y)),
\]
for each $X,Y\in \mc{C}$ which are appropriately associative and unital. Unless otherwise stated, all functors between spectral categories are assumed to be spectral.
\end{definition}

\begin{warning}
    When defining categories, functors, etc. we use the point-set model of $(-)\wedge(-)$ and $\s(-,-)$ in order to have the correct categorical behavior. Outside these settings, we implicitly derive these functors. 
\end{warning}

There is a natural notion of weak equivalence between spectral categories, see e.g.,~\cite[Definition 5.1]{Tabuada} or~\cite[Definition A.3.2.1]{HTT}. 

\begin{definition}
A spectral functor $F: \mc{C} \to \mc{D}$ is a \emph{Dwyer-Kan equivalence} if 
\begin{enumerate}
    \item for every pair of objects $X,Y \in \mc{C}$, the induced map
    \[
    \mc{C}(X,Y) \longrightarrow \mc{C}(F(X), F(Y)) 
    \]
    is a weak equivalence in $\s$;
    \item the induced functor
    \[
    \pi_0(F) : \pi_0(\mc{C}) \to \pi_0(\mc{D})
    \]
    is essentially surjective. 
\end{enumerate}
\end{definition}

A key example of spectral categories the category of functors $\Fun(\mathcal{C},\mathcal{D})$ between spectral categories $\mathcal{C}$ and $\mathcal{D}$ which naturally forms a spectral category with objects the spectral functors $\mathcal{C}\rightarrow \mathcal{D}$ and mapping spectra given by the enriched end
\[
\nat(F,G)\coloneq \int_{c \in \mc{C}} \mathcal{D}(F(c),G(c)).
\]

\subsection{Spectral presheaves} The category of spectra is itself a spectral category, with the mapping spectrum $\s(X,-)$ defined to be right adjoint to $- \wedge X$. Given a spectral category $\mc{C}$, a \textit{spectral presheaf} is a spectral functor $\mc{C}^\op \to \s$. Under the analogy between spectral categories and ring spectra, see e.g.,~\cite{SchwedeShipleyStable}, spectral presheaves are analogous to right modules over a ring spectrum. We find this language compelling in our study of Koszul duality and functor calculus.

\begin{definition}
Let $\mc{C}$ be a spectral category. The category $\sf{RMod}_\mc{C}$ of \emph{right $\mc{C}$-modules} is the spectral category of spectral presheaves on $\mc{C}$:
\[
\sf{RMod}_\mc{C} \coloneq \Fun(\mc{C}^\op, \s).
\]
\end{definition}

The category of right $\mc{C}$-modules comes with a projective model structure, which is characterized by having fibrations and weak equivalences determined levelwise. Since $\s$ has all objects fibrant, all spectral presheaves are also fibrant in the projective model structure. Since homotopy (co)limits are computed objectwise, this model structure is stable. 

\begin{prop} \label{prop: day convolution monoidal model}
Let $\mc{C}$ be a small spectral category. The projective model structure on the category $\mathsf{RMod}(\mc{C})$ of right $\mc{C}$-modules exists.
\end{prop}

The existence of the $\s$-enriched projective model structure is due to \cite[Theorem 6.1]{Schwede_Shipley_2003} and \cite[Proposition 2.4 and Theorem 4.32]{may_guillou}.

A rather surprising result is that even without imposing cofibrancy conditions on spectral model categories, their right module categories are extraordinarily well-behaved. The following is proven in ~\cite[Proposition 2.4]{may_guillou}. 

\begin{prop}\label{prop:quillen}
Given a spectral functor $f:\mathcal{C} \rightarrow \mathcal{D}$, there is a spectrally enriched Quillen adjunction 
\begin{align*}
\adjunction{\ind_{f^\op}}{\sf{RMod}_\mc{C}}{\sf{RMod}_\mc{D}}{\res_{f^\op}}.
\end{align*}
If $f$ is a Dwyer--Kan equivalence, this is an enriched Quillen equivalence.
\end{prop}

\begin{prop}\label{prop: model structure on modules with action}
Let $\mc{C}$ and $\mc{D}$ be spectral categories. The projective model structure on the category $\Fun(\mc{C}, \mathsf{RMod}_\mc{D})$ exists.
\end{prop}
\begin{proof}
This follows from~\cite[Theorem 4.32]{may_guillou}.
\end{proof}

\begin{definition}\label{def: pseudo-cofibrant}
An object $c$ of a symmetric monoidal model category $(\mc{C},\otimes)$ is \textit{pseudo-cofibrant} if $c \otimes -$ preserves cofibrations. 
\end{definition}

When $\mc{C}$ satisfies the monoid axiom (cf.~\cite[Definition 3.3]{SchwedeShipleyAlgebras}), as $\s$ does, $c \in \mc{C}$ being pseudo-cofibrant implies that  $c \otimes -$ is left Quillen. Examples in $\s$ include cofibrant spectra and $\Sigma^\infty X$ for $X$ a CW-complex.

\begin{lem}\label{lem:modules with action cofibrancy}
    If $\mc{C}$ has a single object $\ast$ and the spectrum $\mc{C}(\ast,\ast)$ is pseudo-cofibrant, the forgetful functor given by evaluation
    \[\mathsf{Fun}(\mc{C}, \mathsf{RMod}_\mc{D}) \rightarrow \mathsf{RMod}_\mc{D}\]
  preserves cofibrations and weak equivalences.

\end{lem}

\begin{proof}
    The statement about weak equivalences is immediate. 
   
We identify $Q,R \in \mathsf{Fun}(\mc{C}, \mathsf{RMod}_\mc{D})$ with $\mc{D}$-modules with an action of $\mc{C}(\ast,\ast)=: A$. For a $\mc{D}$-module $T$ with an action of $A$, there is an isomorphism
    \[\mathsf{Fun}(\mc{C}, \mathsf{RMod}_\mc{D})(Q,T) \cong \mathsf{RMod}_\mc{D}(Q,\mathsf{Spec}(A,T))\]
    since $\mathsf{Spec}(A,T(c))$ is the coinduced $A$-module spectrum on $T(c)$.
    Suppose $Q \hookrightarrow R$ is a cofibration of $\mc{D}$-modules with an action of $A$. 
    
    The lifting problem in the category $\mathsf{RMod}_\mc{D}$
\begin{center}
\begin{tikzcd}
Q \arrow[d] \arrow[r] & T \arrow[d, "\simeq", two heads] \\
R \arrow[r]           & S                               
\end{tikzcd}
\end{center}
has a solution, if and only if the lifting problem in $\mathsf{Fun}(\mc{C}, \mathsf{RMod}_\mc{D})$
\begin{center}
\begin{tikzcd}
Q \arrow[d,hook] \arrow[r] & {\mathsf{Spec}(A,T)} \arrow[d] \\
R \arrow[r]           & {\mathsf{Spec}(A,S)}          
\end{tikzcd}
\end{center}
has a solution. Since $R$ is cofibrant as a $\mc{D}$-module with an $A$ action, to show that $R$ is cofibrant as a $\mc{D}$-module, it suffices to show that the map
\[\mathsf{Spec}(A,T) \rightarrow \mathsf{Spec}(A,S)\]
is an acyclic fibration. This is true provided that $\mathsf{Spec}(A,-)$ preserves acyclic fibrations which follows formally from the fact that $A$ is pseudo-cofibrant.
\end{proof}

\subsection{Example: Borel equivariant $G$-spectra}
Let $G$ be a topological group. The category of Borel $G$-spectra $\mathsf{Spec}^{BG}$ is the category of enriched functors $\Fun(G,\mathsf{Spec})$, where $G$ is the spectral category consisting of a single object with endomorphism spectrum given by $\Sigma^\infty_+G$, together with its projective model structure. A standard adjunction argument shows that this is equivalent to the standard description of spectra with a continuous $G$-action. 

Recall that the homotopy fixed points functor
\[
(-)^{hG} : \s^{BG} \longrightarrow \s\]\[X \longmapsto X^{hG} = \mathsf{Spec}(\Sigma^\infty_+ EG,X)^G
\]
is the derived functor of the fixed points functor. Similarly, the homotopy orbits functor 
\[
(-)_{hG} :  \s^{BG} \longrightarrow \s\]\[X \longmapsto X_{hG} = (\Sigma^\infty_+ EG \wedge -)_G
\]
is the derived functor of the orbits functor. These functors are related by the \emph{norm map}. To define the norm for $G$ a general topological group, we must introduce the dualizing spectrum of $G$.

\begin{definition}
For $G$ a topological group, the \emph{dualizing spectrum} $D_G$ is the Borel $G$-spectrum $\Sigma^\infty_+ G ^{hG}$.
\end{definition}

We summarize some main results of \cite{klein_2001} regarding the existence and properties of the norm map. These properties are tied to the existence of a ``six functor formalism'' on the $\infty$-category of spaces, see e.g.,~\cite[\S1.4]{NikolausScholze}.

\begin{thm}[Klein] \label{thm:norm for g spectra}
Let $G$ be a topological group.
\begin{enumerate}
    \item For each Borel $G$-spectrum $X$, there exists a natural transformation in $\mathsf{Ho}(\mathsf{Spec}^{BG})$
    \[
    \mathsf{Nm}_G: (X \wedge D_G)_{hG} \longrightarrow X^{hG},
    \]
called the norm map.
 \item If $X$ is weakly equivalent to a complex built out of finitely many free $G$-cells, then the norm map is an equivalence.
 \item If $G$ is a compact Lie group and $X$ is of the form $Z \wedge G$ for an arbitrary $G$-spectrum $Z$, then the norm map is an equivalence.
 \item If $G$ is a compact Lie group, then $D_G$ is $S^{\mathsf{Ad}_G}$, the one-point compactification of the adjoint representation of $G$. In particular, $D_G$ is invertible as a Borel $G$-spectrum with inverse $S^{- \mathsf{Ad}_G}$.
\end{enumerate}
\end{thm}

\begin{rem}
Finite discrete groups are zero dimensional Lie groups, hence their dualizing spectrum is $S^0$ with the trivial action, and the norm map recovers the more classical norm map $(-)_{hG} \to (-)^{hG}$.
\end{rem}

Klein later characterized the homotopy cofiber of the norm map  
\[
(X \wedge D_G)_{hG} \longrightarrow X^{hG}
\]
as the universal excisive approximation of the homotopy fixed points $(-)^{hG}$ \cite[Theorem A]{klein_2002}. Using this, Kuhn~\cite[Proposition 2.3]{kuhn_2004} gave a simple argument that for finite, discrete $G$ any homotopy natural transformation
\[
(-)_{hG} \longrightarrow (-)^{hG},
\]
which is an equivalence on $\Sigma^\infty_+ G$ must be the norm map, up to an automorphism of $(-)_{hG}$. We repeat the argument for general $G$.

\begin{prop}
For a topological group $G$, any homotopy natural transformation
    \[
    (- \wedge D_G)_{hG} \longrightarrow (-)^{hG},
    \]
    which is an equivalence on $\Sigma^\infty_+ G$ agrees with the norm map, up to an automorphism of $ (- \wedge D_G)_{hG}$.
\end{prop}
\begin{proof}
    Observe that $(- \wedge D_G)_{hG}$  preserves colimits and is equipped with a given natural transformation 
    \[
    (- \wedge D_G)_{hG} \rightarrow (-)^{hG}.
    \]
    The cofiber of this map then satisfies the axioms of \cite{klein_2002}, and so we get induced maps

\begin{center}
\begin{tikzcd}
(X \wedge D_G)_{hG} \arrow[r] \arrow[d, "\phi(X)"'] & X^{hG} \\
(X \wedge D_G)_{hG} \arrow[ru, "\mathsf{Nm}_G"']    &       
\end{tikzcd}
\end{center}
    
   By hypothesis, these maps are equivalences when $X=\Sigma^\infty_+ G$. Since $\Sigma^\infty_+ G$ generates the category of Borel $G$-spectra under homotopy colimits, we conclude the result.
\end{proof}

\begin{rem}
If the dualizing spectrum $D_G$ is invertible, i.e., has the nonequivariant homotopy type of a sphere, we will refer to the natural transformation
\[
(-)_{hG} \longrightarrow ((D_G)^\vee \wedge -)^{hG},
\]
as the dual norm map and observe that it is also unique.

\end{rem}

\subsection{Example: Right modules over operads}\label{subsection:right modules over operad}
The category of symmetric sequences of spectra
\[\mathsf{SymSeq}(\mathsf{Spec})\coloneq \mathsf{Fun}(\mathsf{FinSet}^\cong,\mathsf{Spec})\]
admits a \emph{composition product} $\circ$ given by 
\[
(P \circ Q)(n) = \bigvee_{k=1}^{n}~\bigvee_{n_1 + \cdots + n_k = n} \Sigma_n \wedge_{\Sigma(n_1, \dots, n_k)} P( k) \wedge Q(n_1) \wedge \cdots Q(n_k),
\]
where we assume $n_1 \geq n_2 \geq \cdots \geq n_k > 0$, and $\Sigma(n_1, \dots, n_k)$ is the normalizer of $\Sigma_{n_1} \times  \cdots \times \Sigma_{n_k}$ in $\Sigma_n$, or equivalently, $\Sigma(n_1, \dots, n_k)$ is the semi-direct product of $\Sigma_{n_1} \times  \cdots \times \Sigma_{n_k}$ in $\Sigma_n$ with the subgroup of $\Sigma_k$ which permutes the indices $i$ for which the corresponding $n_i$ are equal. The category of operads (in spectra) is given by the category of monoids for the composition product, and the category $\mathsf{RMod}_P$ of right modules over a given operad $P$ may be defined as the category of symmetric sequences $R$ equipped with composition maps 
\[
R \circ P \rightarrow R.
\]
which make $R$ into a right $P$-module with respect to $\circ$ in the categorical sense.
We describe an alternative characterization of operads and right modules, see e.g.,~\cite[\S4]{AroneTurchin},~\cite[Appendix A]{ACOperads} or~\cite[\S10.1]{MayZhangZou} for a more historical account. To an operad $P$, let $\mathsf{Env}(P)$ denote the free symmetric monoidal spectral category subject to the following constraints: 
\begin{enumerate}
    \item The objects are the finite sets equipped with disjoint union as a symmetric monoidal product.
    \item The mapping spectrum $\mathsf{Env}(P)(n,1)=P(n)$, and in general, 
    \[
    \mathsf{Env}(P)(n,k) = \bigvee_{\text{ordered partitions of} \ \{1, \dots, n\}} P(n_1) \wedge \cdots \wedge P(n_k).
    \]
\end{enumerate}
There is an equivalence of categories 
\[
\mathsf{RMod}_{\mathsf{Env}(P)} \coloneq \mathsf{Fun}(\mathsf{Env}^\mathsf{op}(P),\mathsf{Spec}) \cong \mathsf{RMod}_P.
\]
Thus, our statements about right modules over $\mc{C}$ can be understood as generalizations of statements about right modules over operads.

\section{Augmented categories and topological André--Quillen homology} \label{section: TAQ}

\subsection{Augmented categories}  In this section, we introduce a class of categories whose right modules automatically have an ``indecomposables'' functor.

\begin{definition}
    Given a spectral category $\mathcal{C}$, we let $\mathsf{End}(\mathcal{C})$ denote the  spectral category with objects $\mathsf{Ob}(\mathcal{C})$ and morphisms
    \[
       \mathsf{End}(\mathcal{C})(c,d) = 
    \begin{cases}
        \mathcal{C}(c,c) & \text{if} \ c=d \\
        \ast & \text{otherwise.}
    \end{cases}
    \]
\end{definition}

\begin{definition} 
    An \emph{augmented spectral category} $\mathcal{C}$ is a spectral category $\mathcal{C}$ with a functor $ \epsilon: \mathcal{C} \rightarrow \mathsf{End}(\mathcal{C})$ such that the composite
    \[\mathsf{End}(\mathcal{C}) \xrightarrow{\ \eta \ } \mathcal{C} \xrightarrow{\ \epsilon \ } \mathsf{End}(\mathcal{C})\]
    is the identity, where $\eta$ is the canonical map given by the identity on objects and automorphisms. Note that augmentations are unique when they exist.
\end{definition}

\begin{exs}\hspace{10ex}\label{ex: augmented cats}
\begin{enumerate}
\item  The envelope $\mathsf{Env}(P)$ of a reduced operad $P$ is an augmented category. Its endomorphism category is the symmetric groupoid $\Sigma_\ast$ defined to have objects the natural numbers and morphism spectra
\[
\Sigma_\ast(c,d) = 
\begin{cases}
        \Sigma_+^\infty \Sigma_c & \text{if $c=d$} \\
        \ast & \text{otherwise}.
    \end{cases}
    \]

\item The spectral category $\mathsf{OEpi}:= \mathsf{Vect}_\mathbb{R}^\mathsf{op}$ of finite-dimensional inner product spaces and ``orthogonal surjections'' is an augmented category. Its endomorphism category is the orthogonal groupoid $O(\ast)$ defined to have objects the natural numbers and morphism spectra
\[
O(\ast)(c,d) = 
\begin{cases}
        \Sigma_+^\infty O(c)& \text{if $c=d$} \\
        \ast & \text{otherwise.}
    \end{cases}
    \]
    \end{enumerate}
\end{exs}

\begin{definition}\label{def: A* category}
Let $A_\ast$ be a natural number indexed  sequence $\{A_i\}_{i \in \bb{N}}$ of ring spectra. An $A_\ast $-category is a spectral category $\mathcal{C}$ with pseudo-cofibrant mapping spectra with the property that if $n < m$, then $\mathcal{C}(n,m)=\ast $ and an isomorphism $\mathsf{End}(\mathcal{C}) \cong A_\ast$.
\end{definition}

It is immediate from the condition that $\mathcal{C}(n,m)=\ast $ if $n<m$ that augmentations exist for $A_\ast$-categories, hence right $\mc{C}$-modules have underlying right $A_\ast$-modules, which we routinely refer to as $A_\ast$-sequences. The cofibrancy assumptions on $\mc{C}$ are rather minor and can always be met, up to replacing $\mc{C}$ by a Dwyer-Kan equivalent category \cite[Corollary 7.14]{muro}. A functor of $A_\ast$-categories is a spectral functor of underlying categories which commutes with the augmentation.

\begin{definition}
    A Dwyer Kan equivalence of $A_\ast$-categories is a functor of $A_\ast$-categories that is a Dwyer-Kan equivalence.
\end{definition}

\subsection{Topological André--Quillen homology of right modules}\label{subsection: TAQ} Our attention now turns to understanding the indecomposables of right modules over $A_\ast$-categories. For this, fix an $A_\ast$-category $\mc{C}$.

\begin{definition}\label{def: freeforget}
Let $\mc{C}$ be an $A_\ast$-category. The \emph{free-forgetful adjunction} is the spectrally enriched adjoint pair
\[
\adjunction{\sf{Free}_\mc{C}}{\sf{RMod}_{A_\ast}}{\sf{RMod}_\mc{C}}{\sf{Res}_\mc{C}},
\]
induced by the functor $\eta: A_\ast \to \mc{C}$. We will say that a right $\mc{C}$-module is \emph{free} if it is in the image of the free functor.
\end{definition}

By definition, the free functor $\sf{Free}_\mc{C}$ is the (enriched) left Kan extension along the functor $\eta^\op: A_\ast^\op \to \mc{C}^\op$. In general this left Kan extension is given by the coend
\[
\sf{Free}_\mc{C}(X)(m) = \int^{n \in  A_\ast} \mc{C}(m, \eta(n)) \wedge X(n) \cong \bigvee_n X(n) \wedge_{A_n} \mathcal{C}(m,n)
\]
using the fact that $A_\ast$ has no nontrivial non-automorphisms.

\begin{definition}\label{def: indecomtriv}
Let $\mc{C}$ be an $A_\ast$-category. Define the \emph{indecomposables-trivial adjunction} to be the spectrally enriched adjoint pair
\[
\adjunction{\sf{Indecom}_{\mc{C}}}{\sf{RMod}_\mc{C}}{\sf{RMod}_{A_\ast}}{\sf{Triv}_\mc{C}},
\]
induced by the functor $\epsilon: \mc{C} \to A_\ast$. We will say that a right $\mc{C}$-module is \emph{trivial} if it is in the image of the trivial functor. Explicitly, the indecomposables are given via the coend formula,
\[
\mathsf{Indecom}_\mc{C}(R)(m) \cong \int^{n \in \mathcal{C}}  A_m(n) \wedge R(n).
\]
where we treat $A_m$ as an $A_\ast$-sequence concentrated in degree $m$.
\end{definition}

\begin{definition}
Let $\mc{C}$ be an $A_\ast$-category. The topological André--Quillen homology of a right $\mc{C}$-module $R$ is the $A_\ast$-sequence
\[
\mathsf{TAQ}(R)\coloneq\mathsf{Indecom}_\mc{C}^{\mbf{L}}(R).
\]
\end{definition}

In the case $\mathcal{C}$ is the opposite of the envelope of a reduced operad $P$, minor point-set conditions on $R$ and $P$ imply $\mathsf{TAQ}(R) \simeq B(R,P,1)$, the operadic bar construction, see e.g.,~\cite[\S2.1]{ACOperads}.

\begin{prop}\label{prop:taq free}
Let $\mc{C}$ be an $A_\ast$-category. If $X$ is a right $A_\ast$-module, then
 \[\sf{Indecom}_\mc{C}(\sf{Free}_\mc{C}(X))\cong X.\]
If $X$ is a cofibrant $A_\ast$-sequence, then 
\[
\mathsf{TAQ}(\mathsf{Free}_\mc{C}(X)) \cong X.
\]
\end{prop}
\begin{proof}
If $X$ is cofibrant, $\sf{Free}_\mc{C}(X)$ is a cofibrant right $\mc{C}$-module and hence 
\[
\sf{TAQ}(\sf{Free}_\mc{C}(X)) = \sf{Indecom}_\mc{C}(\sf{Free}_\mc{C}(X)),
\]
so we've reduced to the point-set statement. The augmentations of $\mc{C}$
\[
A_\ast \xrightarrow{\ \eta \ } \mc{C} \xrightarrow{ \ \varepsilon \ } A_\ast, 
\]
compose to the identity, and hence by Definition~\ref{def: freeforget} and Definition~\ref{def: indecomtriv} we have
\[
\sf{Indecom}_\mc{C}(\sf{Free}_\mc{C}(X)) = \ind_{\varepsilon^\op} \circ \ind_{\eta^\op}(X) \cong \ind_{\varepsilon^\op \circ \eta^\op}(X) = \ind_{\id}(X) = X. \qedhere
\]
\end{proof}

\subsection{(Co)Filtration of right modules}\label{subsection: filtration of right modules} The structure of an $A_\ast$-category allows us to define (co)truncations for right modules, which still live in the category of right $A_\ast$-modules. Though simple, these (co)truncations are vital to our study of Weiss calculus, since we ultimately identify them as a reflection of the Goodwillie filtration. Fix an $A_\ast$-category $\mathcal{C}$ and $R \in \mathsf{RMod}_\mathcal{C}.$ 

\begin{definition}
Let $\mc{C}$ be an $A_\ast$-category and $R$ a right $\mc{C}$-module. For each $n \in \bb{N}$, we define the \emph{$n$-th truncation} $R^{\leq n}$ by
\[
R^{\leq n}(m)=
\begin{cases}
R(m) & \text{if} \ m \leq n\\
\ast & \text{if} \ m > n. \\
\end{cases}
\]
Analogously, we define the \emph{$n$-th cotruncation} $R^{\geq n}$ of a right $\mc{C}$-module $R$ by
\[
R^{\geq n}(m) =
\begin{cases}
R(m) & \text{if} \ m \geq n \\
\ast & \text{if} \ m < n.
\end{cases}
\]
The (co)truncation of a right $\mc{C}$-module naturally has the structure of a right $\mc{C}$-module. \footnote{We will let the reader take a wild guess at what we mean by $R^{<n}$ and $R^{>n}$.}
\end{definition}

\begin{lem}\label{lem:triv fiber sequence}
Let $\mc{C}$ be an $A_\ast$-category and $R$ a right $\mc{C}$-module. There is a fibration 
\[
R^{>n} \longrightarrow R \longrightarrow R^{\leq n},
\]
natural in $R$,  as well as natural isomorphisms 
\[
(R^{\leq n})^{\geq n}\cong(R^{\geq n})^{\leq n} \cong \mathsf{Triv}_\mathcal{C}(R(n)),
\]
where we treat $R(n)$ as an $A_\ast$-sequence concentrated in degree $n$
\end{lem}
\begin{proof}
    Fibrations in the projective model structure are given by pointwise fibrations, and the involved maps are either identities or maps to the point. Since all spectra are fibrant we conclude the result.
\end{proof}

The truncations of a right $\mc{C}$-module $R$ assemble into a tower
\[\begin{tikzcd}
	&&& R \\
	\cdots & {R^{\leq n+1}} & {R^{\leq n}} & {R^{\leq n-1}} & \cdots & {R^{\leq 1}} & {R^{\leq 0}}
	\arrow[curve={height=6pt}, from=1-4, to=2-2]
	\arrow[curve={height=6pt}, from=1-4, to=2-3]
	\arrow[from=1-4, to=2-4]
	\arrow[curve={height=-6pt}, from=1-4, to=2-6]
	\arrow[curve={height=-6pt}, from=1-4, to=2-7]
	\arrow[from=2-1, to=2-2]
	\arrow[from=2-2, to=2-3]
	\arrow[from=2-3, to=2-4]
	\arrow[from=2-4, to=2-5]
	\arrow[from=2-5, to=2-6]
	\arrow[from=2-6, to=2-7]
\end{tikzcd}\]
under $R$, which we call the \emph{truncation tower of $R$}.

\begin{lem}\label{lem: module limit of truncations}
Let $\mc{C}$ be an $A_\ast$-category and $R$ a right $\mc{C}$-module. The comparison map
    \[
    R \longrightarrow \lim_n R^{\leq n},
    \]
    is an isomorphism and the inverse limit is a tower of fibrations, and so is a homotopy inverse limit.
\end{lem}
\begin{proof}
    What remains to be checked is that the map is an isomorphism. Since limits in functor categories are computed pointwise, it suffices to check
    \[
    R(m) \xlongrightarrow{\cong} \lim_n R^{\leq n}(m).
    \]
    After $n>m$, this limit is constant at $R(m)$, and so we deduce the result.
\end{proof}

\begin{prop}
Let $\mc{C}$ be an $A_\ast$-category with right modules $R$ and $S$. The map
\[
\mathsf{RMod}_\mathcal{C}(R,S) \longrightarrow \lim_{n} \mathsf{RMod}_\mathcal{C}(R,S^{\leq n}),
\]
induced by the truncation tower of $S$ is an isomorphism. If $R$ is cofibrant, then this is a tower of fibrations and so the inverse limit is a homotopy inverse limit.
\end{prop}
\begin{proof}
    This follows from Lemma \ref{lem: module limit of truncations} and the axioms of a spectral model category.
\end{proof}

We refer to this as the truncation filtration of the right module mapping spectrum, and turn our attention to the layers of this filtration.

\begin{prop}\label{prop: layers of truncation tower}
Let $\mc{C}$ be an $A_\ast$-category with right modules $R$ and $S$. There is a fiber sequence
\[
\mathsf{RMod}_{A_n}(\mathsf{Indecom}(R)(n), S(n)) \longrightarrow \mathsf{RMod}_\mathcal{C}(R,S^{\leq n}) \longrightarrow \mathsf{RMod}_\mathcal{C}(R,S^{\leq n-1})
\]
natural in $R$ and $S$.
\end{prop}
\begin{proof}
    This follows from applying $\mathsf{RMod}_\mathcal{C}(R,-)$ to the fiber sequence 
    \[
    \mathsf{Triv}_\mathcal{C}( S(n)) \rightarrow S^{\leq n} \rightarrow S^{\leq n-1}
    \]
    of Lemma \ref{lem:triv fiber sequence} and applying the indecomposables-trivial adjunction of Definition~\ref{def: indecomtriv}.
\end{proof}

\begin{cor}
Let $\mc{C}$ be an $A_\ast$-category with right modules $R$ and $S$. There is an equivalence,
    \[\mathsf{RMod}_\mc{C}^h(R,S) \simeq \underset{n}{\holim}~ \mathsf{RMod}^h_\mathcal{C}(R,S^{\leq n}),\]
    and the $n$-th layer of the tower of fibrations is
    $\mathsf{RMod}^h_{A_n}(\mathsf{TAQ}(R)(n), S(n))$.
\end{cor}

\subsection{Properties of topological Andr{\'e}-Quillen homology}
The fundamental observation which makes topological André--Quillen homology interesting is that even though there is an isomorphism of symmetric sequences,
\[\mathsf{Indecom}(\mathsf{Triv}_\mathcal{C}( A_n)) \cong A_n,\]
the derived statement is far from the truth
\[\mathsf{TAQ}(\mathsf{Triv}_\mathcal{C}( A_n)) \not\simeq A_n.\]

Nevertheless, the naive statement is true in a range. In order to verify this, we introduce the ideal of cell structures for right $\mc{C}$-modules. We call $\mathsf{Free}_\mc{C}( \bb{S}_c^m \wedge A_n)$ an $(m,n)$-cell, $e_{(m,n)}$. A \textit{cell complex} is a right $\mc{C}$-module built from, possibly infinite, iterated pushouts of cells.

We will say a right module $R$ over an $A_\ast$-category $\mc{C}$ is \textit{homotopically concentrated above degree $n-1$} if the natural map $R^{\geq n} \rightarrow R$ is an equivalence. 

\begin{prop}\label{prop: cell complex structure}
Let $\mc{C}$ be an $A_\ast$-category and $R$ a right $\mc{C}$-module. If $R$ is homotopically concentrated above degree $n-1$, then it is equivalent to a cell complex built out of $(m,l)$-cells where $l\geq n$ and $m$ can vary.

\end{prop}
\begin{proof}
    There are two key facts:
    \begin{itemize}
        \item There is an equivalence $\mathsf{RMod}_\mc{C}(e_{(m,n)},R) \simeq \s^h(\bb{S}^m,R(n))$.
        \item There are isomorphisms $e_{(m,n)}(j)=\ast$ for $j<n$ and $e_{(m,n)}(n)=\bb{S}_c^m \wedge A_n$.
    \end{itemize}
Thus, from a cellular replacement of $R(n)$ in $\mathsf{RMod}_{A_n}$ we may iteratively attach $e_{(m,n)}$, as $m$-varies, to build a cell complex with the correct spectrum in degree $n$. The first fact then implies we can do this inductively, without affecting lower terms.
\end{proof}

\begin{lem}\label{lem:taq concentrate}
Let $\mc{C}$ be an $A_\ast$-category and $R$ a right $\mc{C}$-module. If $R$ is homotopically concentrated above degree $n-1$, the canonical map 
    \[
    R^{\leq n} \rightarrow \mathsf{TAQ}(R)^{\leq n}
    \]
    given by the $n$-truncation of the derived counit of the ($\mathsf{Indecom}$,$\mathsf{Triv}$)-adjunction is an equivalence.

    More generally, the map $R \rightarrow R^{\leq n}$ induces an equivalence after applying $\mathsf{TAQ}(-)^{\leq n}$,
\end{lem}
\begin{proof} 
 The result follows from Proposition \ref{prop:taq free} and Proposition \ref{prop: cell complex structure} since $\mathsf{TAQ}$ commutes with homotopy colimits.
\end{proof}

\begin{cor}\label{cor:taq of triv}
Let $\mc{C}$ be an $A_\ast$-category. There are equivalences 
\[
\mathsf{TAQ}(\mathsf{Triv}_\mathcal{C}( A_n))(m) \simeq
\begin{cases}
    A_n & \text{if \ $n=m$,} \\
    \ast & \text{if \ $n >m$.}
\end{cases}
\]
\end{cor}

In general, there are infinitely many nontrivial spectra making up $\mathsf{TAQ}(\mathsf{Triv}_\mathcal{C}( A_n))$ for fixed $n$.

\begin{ex}
    If $\mc{C}$ is $\mathsf{Surj}$, the category of finite sets and surjections, then $\mathsf{TAQ}(\mathsf{Triv}(\Sigma_1))$ may be identified with the Spanier--Whitehead duals of the partition poset complexes, or equivalently, the Lie cooperad \cite[Lemma 8.6]{ChingOperad}.
\end{ex}

Recall that we say a spectral category $\mc{C}$ is locally finite if the mapping spectra are finite. In the case of an $A_\ast$-category we introduce the following variant.

\begin{definition} \label{def:locally finite category}
 An $A_\ast$-category $\mc{C}$ is \textit{locally right finite} if $\mc{C}(i,j)$ is homotopy equivalent to an $A_j$-spectrum built out of finitely many free cells.
\end{definition}

\begin{prop}\label{prop:taq of finite}
Let $\mc{C}$ be an $A_\ast$-category and $R$ a right $\mc{C}$-module. If $\mc{C}$ is locally right finite and $R$ is level-finite, then $\mathsf{TAQ}(R)$ is level-finite.
\end{prop}

\begin{proof}

 Consider the homotopy fiber sequence \[F_1 \rightarrow \mathsf{Free}^\mbf{L}_\mathcal{C}(R) \rightarrow R .\]
Iterating this process we define the higher relations:
 \[F_{n+1} \rightarrow \mathsf{Free}^\mbf{L}_\mathcal{C}(F_{n}) \rightarrow F_n .\]
Note all the $F_n$ are level-finite by our finiteness assumption and $F_{n+1}$ is homotopically concentrated above degree $n$ by induction.  Applying $\mathsf{TAQ}$ to this homotopy (co)fiber sequence, we obtain another homotopy (co)fiber sequence
   \[\mathsf{TAQ}(F_{n+1}) \rightarrow F_n \rightarrow \mathsf{TAQ}(F_n).\]

   Because of how these cofiber sequences are connected, if we know that $\mathsf{TAQ}(F_{n+1})(m)$ is finite then $\mathsf{TAQ}(F_{n})(m)$ is finite, and by repeated application, $\mathsf{TAQ}(R)(m)$ is finite. To deduce the result, pick $n>m$ and obverse that $TAQ(F_{n-1})(m)$ is finite by Lemma \ref{lem:taq concentrate}. 
\end{proof}

\section{Koszul duality for augmented categories}\label{section:koszul duality}

The topological André--Quillen homology of an algebro-topological object $A$ is defined in great generality as the left derived functor of quotienting $A$ by its decomposable elements. Hence, for cofibrant $A$ one has
\[\mathsf{TAQ}(A)\simeq A/\mathsf{Decom}(A).\]

As a consequence, one finds that free objects behave well with respect to $\mathsf{TAQ}$:
\[\mathsf{TAQ}(\mathsf{Free}(X))\simeq X.\]
Hence, if we express $A$ in terms of free cells, $\mathsf{TAQ}(A)$ obtains a cell structure with cells in bijective correspondence with those of $A$. In this sense, $\mathsf{TAQ}(A)$ is telling us homological information about $A$ from the point of view of the ambient category. 

A natural question to ask is if $\mathsf{TAQ}$ detects equivalences. That is, if $A \rightarrow B$ induces an equivalence $\mathsf{TAQ}(A) \xrightarrow{\simeq} \mathsf{TAQ}(B)$, do we necessarily have $A \xrightarrow{\simeq} B$? One approach to this question is to build a cellular spectral sequence to compute the homology of $A$ in terms of the homology of $\mathsf{TAQ}(A)$ and the homology of free algebras, see ~\cite{GKR} for the case of $E_n$-algebras. In general, there are obstructions to $\mathsf{TAQ}(A)$ detecting equivalences. Without connectivity hypotheses, $\mathsf{TAQ}$ often destroys $p$-torsion information. In ~\cite[Theorem 3.4]{Mandell2006} it is demonstrated that $p$-complete $\mathsf{TAQ}$ of $X^\vee$ is contractible for finite type, nilpotent spaces $X$.

Koszul duality is motivated by the question of what symmetries the indecomposables of an algebraic object have, and under what conditions do these symmetries allow us to recover our initial objects. When these conditions are satisfied, then topological André--Quillen homology necessarily detects weak equivalences. This problem was largely solved in \cite{chingHarper} for algebras over an operad, see also recent work of Heuts~\cite{HeutsKoszul}.

\subsection{The Koszul dual of an augmented category} 

Given an $A_\ast$-category $\mathcal{C}$, we now present a model categorical construction of the Koszul dual $K(\mathcal{C})$; a model independent description is discussed in Remark \ref{rem: model independent Koszul dual}. We remind the reader that $A_\ast$-categories (Definition~\ref{def: A* category}) are a class of categories with objects in bijection with $\mathbb{N}$ and morphisms which \textit{oppose} the order $\leq$, along with minor cofibrancy hypotheses.  We first produce a few technical lemmas regarding cofibrant replacements of trivial right $\mathcal{C}$-modules.

\begin{lem}\label{lem:cofibrant models with action}
Let $\mc{C}$ be an $A_\ast$-category. For each $n \in \bb{Z}^{\geq 0}$, there exists a cofibrant right $\mc{C}$-module $T_n$ such that
\begin{enumerate}
    \item the module $T_n$ is a cofibrant model for $\mathsf{Triv}_\mc{C}(A_n)$, where $A_n$ is viewed as an $A_\ast$-sequence concentrated in degree $n$;
    \item there is a canonical map of ring spectra
    \[
    A_n \longrightarrow \mathsf{RMod}_\mc{C}(T_n, T_n)
    \]
    which is an equivalence; and,
    \item if $n>m$, the inclusion of the trivial morphism
    \[
    \ast \longrightarrow \mathsf{RMod}_\mc{C}(T_n, T_m)
    \]
    is an equivalence.
\end{enumerate}
\end{lem}
\begin{proof}
First observe that at the point-set level there are isomorphisms
\[
\mathsf{RMod}_\mathcal{C} (\mathsf{Triv}_\mathcal{C} (A_n),\mathsf{Triv}_\mathcal{C} (A_n)) \cong \mathsf{RMod}_{A_n}(A_n,A_n) \cong A_n.
\]
It follows that $\mathsf{Triv}_\mc{C}(A_n)$ is naturally a right $\mc{C}$-module with an action of $A_n$, i.e., for $A_n$ the spectral category with a single object and endomorphisms given by $A_n$, there is a spectral functor 
\[
A_n \longrightarrow \mathsf{RMod}_\mc{C}\]
\[\ast \longmapsto \mathsf{Triv}_\mc{C}(A_n),
\]
with action on morphisms defined via the above isomorphisms. This category admits a model structure by Proposition~\ref{prop: model structure on modules with action}. Define $T_n$ to be a cofibrant replacement of $\mathsf{Triv}_\mc{C}(A_n)$ in the category of right $\mc{C}$-modules with an action of $A_n$.  To verify $(1)$, it suffices to show that the underlying right $\mc{C}$-module of $T_n$ is cofibrant. This need not be true in general, however we could retroactively use this argument instead on a replacement $\mc{D}$ of $\mc{C}$ with pseudo-cofibrant mapping spectra and this then follows from Lemma~\ref{lem:modules with action cofibrancy}. Finally, we could induce along $\mc{D}\rightarrow \mc{C}$ to yield a cofibrant model, $T_n$ which is trivial by checking the counit of $(\mathsf{Indecom},\mathsf{Triv})$-adjunction.

The map of $(2)$ exists by construction, since the data of an $A_n$-module in $\mathsf{RMod}_\mc{C}$ is exactly a map $A_n \rightarrow \mathsf{RMod}_\mc{C}(T_n,T_n)$.
The map of $(3)$ is, of course, canonical.  The fact that the claimed maps are weak equivalences follows from using Corollary \ref{cor:taq of triv} to compute the fibers of the truncation filtration of $\mathsf{RMod}_{\mc{C}}(T_n,T_m)$ for $n \geq m$.
\end{proof}

\begin{definition}
Let $\mc{C}$ be an $A_\ast$-category. The \emph{Koszul dual} $A_\ast^\mathsf{op}$-category $K(\mc{C})$ has objects identified with $T_i$ for each $i \in \bb{N}$ and mapping spectra
    \[
       K(\mathcal{C})(c,d) = 
    \begin{cases}
        A^\mathsf{op}_c & \text{if} \ c=d \\
        \mathsf{RMod}_\mathcal{C}(T_d,T_c) & \text{d $<$ c} \\
        \ast & \text{otherwise.}
    \end{cases}
    \]
    Composition of mapping spectra is induced by the opposite of composition of right module mapping spectra, and the inclusion $A_n \xrightarrow{\simeq}\mathsf{RMod}(\mathcal{C})(T_n,T_n)$ of Lemma \ref{lem:cofibrant models with action}.
\end{definition}

\begin{rem}\label{rem: model independent Koszul dual}

By Lemma \ref{lem:cofibrant models with action}, the category $K(\mc{C})$ has a model independent description as the opposite of the full subcategory of the spectral $\infty$-category $\mathsf{RMod}_\mc{C}$ on the $\mathsf{Triv}_\mc{C}(A_n)$ for $n \in \mathbb{N}$. The explicit construction above has the added property that it is an honest $A_\ast$-category. The spectral enrichment is essential, since full subcategories of stable $\infty$-categories need not be stable.
\end{rem}

This definition of Koszul duality is inspired by similar operadic definitions \cite{Lurie_2023,malinonepoint}. It is a categorification of the definition of the Koszul dual \cite{priddy} of an augmented differential graded algebra $A \rightarrow k$ : \[K(A):= \mathrm{RMod}_A^h(k).\]

\begin{ex}
    In the case $\mc{C}$ is the envelope of a reduced level-finite operad $P$, we have that 
\begin{align*}
K(\mathsf{Env}(P))(n,1) 
&\simeq \mathsf{RMod}^h_P(\mathsf{Triv}_P(\Sigma^\infty_+ \Sigma_1),\mathsf{Triv}_P(\Sigma^\infty_+ \Sigma_n)) \\
&\simeq \mathsf{RMod}^h_{\Sigma_i}(\mathsf{TAQ}((\mathsf{Triv}_P(\Sigma^\infty_+ \Sigma_1)),\Sigma^\infty_+ \Sigma_n)) \\
&\simeq (\mathsf{TAQ}(\mathsf{Triv}_P(\Sigma^\infty_+ \Sigma_1))(n))^\vee. 
\end{align*} 

This latter spectrum can be computed as the dual of the operadic bar construction $B(1,P,1)(n)$, recovering the underlying spectrum of Ching's \cite{ChingOperad} and Lurie's \cite{HA} Koszul duals.

\end{ex}

Although the uniqueness and invariance under Dwyer-Kan equivalence of the Koszul dual will not play a role in our arguments, it is worthwhile to briefly discuss these properties. We provide sketches of the arguments.

\begin{prop}\label{prop: unique Koszul dual}
The Dwyer-Kan equivalence class of the Koszul dual of an $A_\ast$-category $\mc{C}$ is independent of the choice of cofibrant approximations of  $\mathsf{Triv}_\mc{C}(A_i)$ which satisfy the requirements of Lemma~\ref{lem:cofibrant models with action}. 
\end{prop}

\begin{proof}
We proceed analogously to \cite[Proposition 7.3]{malinonepoint}. Let $T_n$ and $T'_{n}$ be cofibrant approximations of $\mathsf{Triv}_\mc{C}(A_n)$ which both satisfy the requirements of Lemma~\ref{lem:cofibrant models with action}. By cofibrancy there is an equivalence $T_n \to T'_n$ and one can construct an $A_\ast$-category with objects the natural numbers and morphisms
   \[[n,m] := \mathsf{RMod}_\mc{C}(T'_m,T_n)\]
   when $n > m$.
   Morphism composition is given by first composing with $T_n \rightarrow T_m$ then composing right module mapping objects. Precomposition and postcomposition with the equivalences $T_n \rightarrow T'_n$ yield a zigzag of $A_\ast$-equivalences. 
\end{proof}

\begin{prop}
    Let $\mc{C}$ and $\mc{D}$ be  $A_\ast$-categories. If there is a  Dwyer-Kan equivalence of $A_\ast$-categories of  $f: \mathcal{C} \rightarrow \mathcal{D}$,
    then there is a zigzag
    \[
    K(\mathcal{C}) \simeq K(\mathcal{D})
    \]
    of  Dwyer-Kan equivalence of $A^\mathsf{op}_\ast$-categories.
\end{prop}
\begin{proof}
The proof is analogous to that of \cite[Theorem 7.4]{malinonepoint} and proceeds by observing that if $T_n$ is the cofibrant model of the trivial right $\mc{C}$-module on $A_n$ inducing it along a  Dwyer-Kan equivalence of $A_\ast$-categories yields a functor
\[
 K(f) : K(\mc{C}) \rightarrow K(\mc{D}),
 \]
where we choose our cofibrant models for $\mc{D}$ to be $\mathsf{ind}_f T_n$. The cofibrancy of the $T_n$ implies that $K(f)$ is a  Dwyer-Kan equivalence of $A_\ast$-categories.
\end{proof}

\subsection{The Koszul dual of a right module}

We now turn our attention to Koszul duality for right modules over an $A_\ast$-category. This will be the key algebraic input into our analysis of the Weiss derivatives.

\begin{definition}
    Let $\mc{C}$ be an $A_\ast$-category. The \emph{Koszul dual} of a right $\mathcal{C}$-module $R$ is the right $K(\mc{C})$-module given by the restricted representable functor:
    \[
    T_n \mapsto \mathsf{RMod}_\mathcal{C}(R,T_n).
    \]
\end{definition}

\begin{rem}
Because an opposite was introduced in the definition of $K(\mathcal{C})$, the functor $K(R)$ has the correct variance to make it a right module. Additionally, it is contravariantly functorial in $R$, i.e., Koszul duality defines a functor
\[
K: \mathsf{RMod}_\mc{C} \longrightarrow \mathsf{RMod}_{K(\mc{C})}^\op.
\]
\end{rem}
\begin{ex} \label{ex: koszul dual of free}

 By the ($\mathsf{Indecom}$,$\mathsf{Triv}$)-adjunction, for any $A_n$-module $X$, thought of as an $A_\ast$ sequence concentrated in degree $n$, there  is an isomorphism
   \[K(\mathsf{Free}_\mc{C}(X))\cong \mathsf{Triv}_{K(\mc{C})} (\mathsf{RMod}_{A_n}(X,A_n)).\]
   The derived analog of this statement also holds. We have that as $A_\ast$-sequences \[K^\mbf{L}(\mathsf{Free}^\mbf{L}_\mc{C}(X)) \simeq \mathsf{RMod}^h_{A_n}(X,T_\ast(n)),\]
   and the latter is $\mathsf{RMod}^h_{A_n}(X,A_n)$ homotopically concentrated in degree $n$. A right module homotopically concentrated in a single degree $n$, is automatically homotopy equivalent to a trivial module. This is because we have equivalences
   \[R \xrightarrow{ \ \simeq \ } R^{\leq n} \xleftarrow{\ \simeq \ } (R^{\leq n})^{\geq n} \cong \mathsf{Triv}_{\mc{C}} (R(n)).\]
\end{ex}

We now propose a slightly different point of view. By definition, we have that
\[K(R)(n) \cong \int\limits_{m \in \mc{C}} \s(R(m),T_n(m)).\]
If $R$ is cofibrant and level-finite, this is equivalent to the derived coend
\[\int\limits_{m \in \mc{C}} R(i)^\vee \wedge T_n(m).\]
This formula is reminiscent of the discrete Fourier transform of the functor $R$ with respect to the ``characters'' $T_n$. We will see that many facts from Fourier analysis have analogues for Koszul duality. In particular, we can define a transform in the \textit{opposite} direction by
\[
K^{-1}(S)(m) = \int_{T_n \in K(\mc{C})} \s(S(T_n), T_n(m)).
\]

By adjunction, it is straightforward to check the following:

\begin{prop} \label{prop:formula for koszul dual}
There are equivalences
    \[K^\mbf{L}(R)\simeq \mathsf{RMod}^h(R,\mathsf{Triv}_\mc{C}(\Sigma^\infty_+ A_n))  \simeq \mathsf{RMod}^h_{A_n}(\mathsf{TAQ}(R),\Sigma^\infty_+ A_n)\]
  \[(K^{-1})^\mbf{R}(S)\simeq \mathsf{RMod}^h(S,\mathsf{Triv}_{K(\mc{C})}(\Sigma^\infty_+ A_n^\mathsf{op}))  \simeq \mathsf{RMod}^h_{A_n^\mathsf{op}}(\mathsf{TAQ}(S),\Sigma^\infty_+ A_n^\mathsf{op})\]
\end{prop}

\begin{prop}\label{prop: Koszul adjunction}
Let $\mc{C}$ be an $A_\ast$-category. Koszul duality is part of a spectrally enriched Quillen adjunction
\[
\adjunction{K}{\mathsf{RMod}_{\mathcal{C}}}{\mathsf{RMod}_{K(\mathcal{C})}^{\mathsf{op}}}{K^{-1}}.
\]

\end{prop}
\begin{proof}
Using the definitions in terms of ends, one can check this pair of functors is adjoint via the calculus of (co)ends. The axioms of an enriched model category imply that $K$ sends (acyclic) cofibrations of right modules to (acyclic) cofibrations in $\s^\mathsf{op}$. Hence, we deduce that the pair forms an enriched Quillen adjunction. 
\end{proof}

There is a slightly different formula for the right adjoint to Koszul duality which provides an alternative view on why $K^{-1}(S)$ is a presheaf on $\mc{C}$.

\begin{prop}
Let $\mc{C}$ be an $A_\ast$-category. There is a natural isomorphism
\[
K^{-1}(S)(m) \cong \mathsf{RMod}_{K(\mc{C})}(S,K(\mathsf{Free}_{\mc{C}}(A_m))).
\]
\end{prop}
\begin{proof}
To see the required isomorphism we observe that the latter can be written as
    \[\int_{T_n \in K(\mc{C})} \s(S(T_n), K(\mathsf{Free}_{\mc{C}}(A_m))(n)) .\]
    By Example \ref{ex: koszul dual of free} these agree. 
\end{proof}

\subsection{Foundational results for Koszul duality}

In this section, we discuss  to what extent Koszul duality for right modules is an equivalence. As a precursor to the next definition, we encourage the reader to revisit the definition of an $A_\ast$-category in Definition~\ref{def: A* category}.

\begin{definition}\label{def: dualizable category}
Let $G_\ast$ be the suspension spectrum of a topological groupoid. A $G_\ast$-category $\mc{C}$ is said to be \emph{dualizable} if
\begin{enumerate}
    \item it is locally finite and locally right finite (Definition~\ref{def:locally finite category}); 
    \item the dualizing spectra $D_{G_i}$ are invertible, i.e., nonequivariantly spheres of arbitrary dimension.

\end{enumerate}
\end{definition}

These conditions are often satisfied in practice. The conditions on $G_\ast$ are satisfied by groupoids of, possibly discrete, compact Lie groups, e.g., $ \Sigma_\ast = \coprod_{n \geq 0} \Sigma_n$ or $O(\ast)=\coprod_{n \geq 0} O(n)$.

\begin{exs}\hspace{10ex}
    \begin{enumerate}
 \item If $P$ is a level-finite reduced operad\footnote{With mild cofibrancy conditions.}, then $\mathsf{Env}(P)$ is a dualizable $\Sigma_\ast$-category. This is because $\Sigma_i$ freely acts on partitions with $i$ ordered components that arise in the definition of $\mathrm{Env}(P)$ of Section \ref{subsection:right modules over operad}. 
\item The category of orthogonal epimorphisms $\mathsf{OEpi}$ is a dualizable $O(\ast)$-category. The $O(n)$-action on $\mathsf{OEpi}(m,n) \cong \mathsf{Vect}_\bb{R}(\mathbb{R}^n,\mathbb{R}^m)$ is given by precomposition with $O(n)$. This is free since isometries have left inverses. We give a complete proof of this in Proposition~\ref{prop: OEpi dualizable}.
    \end{enumerate}
\end{exs}

\begin{rem}

The proofs of the results in this section require constructions which do not make sense for arbitrary ring spectra $A_i$, but could likely be reproduced under weaker assumptions: the $A_\ast$ admit augmentations, the $A_\ast$ admit bialgebra structures, the $A_\ast$-primitives of $\bb{S}$ are spherical, etc.
\end{rem}

The following lemma allows us to treat $K$ and $K^{-1}$ uniformly.

\begin{lem}
    If $\mc{C}$ is a dualizable $G_\ast$-category, then the category $K(\mc{C})$ is a dualizable $G_\ast$-category. 
\end{lem}

\begin{proof}
    We must check that the object $\mathsf{RMod}_\mc{C}(T_n,T_m)$ is finite as a $G_n$-spectrum, where the action is by the automorphisms of $T_n$ constructed in Lemma \ref{lem:cofibrant models with action}. By Proposition \ref{prop:taq of finite} and Proposition \ref{prop:formula for koszul dual}, this is equivalent to checking that $\mathsf{TAQ}(T_n)(m)$ is a finite $G_n$-spectrum. To check this, resolve $T_n$ by free cells in the category of $G_n$-equivariant right $\mc{C}$-modules. This may be arranged so that at any given categorical degree, there are only finitely many cells added, as a consequence of the local right finiteness of $\mc{C}$ and the fact that $\mc{C}(i,j)=\ast$ if $i <j$. Applying $\mathsf{TAQ}$ then yields a $G_n$-cell structure of $\mathsf{TAQ}(T_n)$ with the same properties.
\end{proof}

For convenience, all functors that appear in the following proposition are implicitly derived.

\begin{prop} \label{prop: koszul dual of free}
Let $\mc{C}$ be a dualizable $G_\ast$-category. If $S$ is a finite $G_n$-spectrum, then there is an equivalence
   \[K(\mathsf{Free}_\mc{C}(S))\simeq \mathsf{Triv}_{K(\mc{C})} (S^\vee \wedge D_{G_n}),\]
where we view $S$ as a right $G_\ast$-module concentrated in a single degree.

Similarly, 
\[K^{-1}(\mathsf{Free}_{K(\mc{C})}(S))\simeq \mathsf{Triv}_{\mc{C}} (S^\vee \wedge D_{G_n}).\]
   \end{prop}

\begin{proof}
    By a direct computation with the free-forgetful adjunction and the finiteness of $S$, Theorem \ref{thm:norm for g spectra} implies the result as $G_\ast$-sequences. However, the action is forced to be homotopically trivial because it is concentrated in a single degree.
\end{proof}

This result agrees with the Yoneda lemma in the case $S= \Sigma^\infty_+G_n$ since the theory of dualizing spectra implies $\Sigma^\infty_+ G_n^\vee \wedge D_{G_n} \simeq \Sigma^\infty_+ G_n$, under the hypotheses of the section. If $G_n$ is a compact Lie group, this is perhaps more well known as Atiyah duality.

For the next proposition, all functors that appear are implicitly derived.

\begin{prop} \label{prop: koszul dual of trivial} 
Let $\mc{C}$ be a dualizable $G_\ast$-category. If $S$ is a finite $G_n$-spectrum, then there is an equivalence,
   \[K(\mathsf{Triv}_{\mc{C}}(S)) \simeq \mathsf{Free}_{K(\mc{C})}(S^\vee \wedge D_{G_n}), \]

where we view $S$ as a right $G_\ast$-module concentrated in a single degree.

Similarly,
   \[K^{-1}(\mathsf{Triv}_\mc{C}(S))\simeq \mathsf{Free}_{\mc{C}}(S^\vee \wedge D_{G_n}). \]
    
\end{prop}

\begin{proof}

We start by assuming $S \simeq \bb{S}^m \wedge \Sigma^\infty_+ G_n$. Since $\mathsf{Triv}_\mc{C}(\bb{S}^m \wedge \Sigma^\infty_+ G_n) \simeq \bb{S}^m \wedge \mathsf{Triv}_\mc{C}(\Sigma^\infty_+ G_n)$, we may pull the $\bb{S}^m$ out of the right module mapping spectrum computing $K(\mathsf{Triv}_\mc{C}(\bb{S}^m \wedge \Sigma^\infty_+ G_n))$ to get
\begin{align*}
    \mathsf{RMod}_{\mc{C}}(\mathsf{Triv}_\mc{C}(\bb{S}^m \wedge \Sigma^\infty_+ G_n),T_\ast) &\simeq  \s(\bb{S}^m,\mathsf{RMod}_{\mc{C}}(\mathsf{Triv}_\mc{C}(\Sigma^\infty_+ G_n),T_\ast)) \\
    &\simeq \s(\bb{S}^m,\mathsf{Free}_{K(\mc{C})}(\Sigma^\infty_+ G_n)).
\end{align*}
This last object is equivalent to 
\[
\mathsf{Free}_{K(\mc{C})}(\bb{S}^{-m}\wedge \Sigma^\infty_+ G_n )\simeq\mathsf{Free}_{K(\mc{C})}(\bb{S}^{-m}\wedge  D_{G_n} \wedge \Sigma^\infty_+ G_n ^\vee ).
\]
The spectrum $\bb{S}^{-m}\wedge \Sigma^\infty_+ G_n^\vee $ is the dual of a free $G_n$-cell, and so when $S$ is a finite $G$-CW complex, we have
\[K(\mathsf{Triv}_\mc{C}(S)) \simeq \mathsf{Free}_{K(\mc{C})}(S^\vee \wedge D_{G_n}).\]
For general finite $S$, the dual of the cellular filtration is an inverse limit, and the above observation is not immediately true. We must pick a free $G_n$-cell structure on $S$ with $m$-skeleton $S^{\leq m}$ and verify that \[\underset{n \to \infty}{\holim} \mathsf{Free}_{K(\mc{C})}((S^{\leq m})^\vee \wedge D_{G_n}) \simeq \mathsf{Free}_{K(\mc{C})}( S^\vee \wedge D_{G_n}) .\]

To this end, pick a cell structure such that $S^{\leq m}$ is a finite spectrum, something which is possible because both $S$ and $G_n$ are finite.

By definition, the left-hand side restricted to categorical degree $j$ is equivalent to \[\underset{m \to \infty}{\holim}(  ((S^{\leq m})^\vee \wedge D_{G_n} \wedge \mc{C}(j,n))_{hG_n}) \]

By the $G_n$-finiteness of $\mc{C}(j,n)$ and finiteness of the other terms, this homotopy orbits is finite. Hence, we may write this homotopy limit as the dual of a homotopy colimit
\[\underset{n \to \infty}{\hocolim}(  (((S^{\leq m})^\vee \wedge D_{G_n} \wedge \mc{C}(j,n))_{hG_n})^\vee)^\vee \]
which by the finiteness of $S$ simplifies to
\[ \underset{m \to \infty}{\hocolim}(  (S^{\leq m} \wedge D_{G_n}^\vee \wedge \mc{C}(j,n)^\vee)^{hG_n})^\vee .\]
By the $G_n$-finiteness of $\mc{C}(j,n)$ and finiteness of the other terms, the dual norm map gives an equivalence with
\[ \underset{m \to \infty}{\hocolim}(  (S^{\leq m} \wedge \mc{C}(j,n)^\vee)_{hG_n})^\vee \]

We can pull out the homotopy orbits and smash product out of the colimit to get
\[(((\underset{m \to \infty}{\hocolim} S^{\leq m})\wedge \mc{C}(j,n)^\vee )_{hG_n})^\vee \simeq ((S \wedge \mc{C}(j,n)^\vee)_{hG_n})^\vee.\]
The $G_n$-finiteness of $\mc{C}(j,n)$ and finiteness of the other terms means the dual norm map gives
\[((S \wedge D_{G_n}^\vee \wedge C(j,n)^\vee)^{hG_n})^\vee \]
These homotopy fixed points are the dual of a homotopy orbits, by the finiteness of the terms, and so this is
\[((S^\vee \wedge D_{G_n} \wedge C(j,n))_{hG_n})^\vee )^\vee.\]
By the $G_n$-finiteness of $C(j,n)$ and the finiteness of the other terms the homotopy orbits are finite, so this is 
\[(S^\vee \wedge D_{G_n} \wedge C(j,n))_{hG_n},\]

the $j$-th categorical degree of $\mathsf{Free}_{K(\mc{C})}( S^\vee \wedge D_{G_n})$. This concludes the proof of the claim for $K$. The proof for $K^{-1}$ is completely analogous.
\end{proof}

\begin{thm}\label{thm:koszul duality for right modules}
For a dualizable $G_\ast$-category, the spectrally enriched Quillen adjunction
\[
\adjunction{K}{\mathsf{RMod}_{\mathcal{C}}}{\mathsf{RMod}_{K(\mathcal{C})}^{\mathsf{op}}}{K^{-1}},
\]
has the property that when restricted to level-finite right modules, the derived functors of $K,K^{-1}$ are inverse and lift $\mathsf{TAQ}^\vee \wedge D_{G_\ast}$.

\end{thm}
\begin{proof}

We first deal with the second assertion. By Proposition~\ref{prop:formula for koszul dual},
\[K^\mbf{L}(R)\simeq \mathsf{RMod}^h(R,\mathsf{Triv}_\mc{C}(\Sigma^\infty_+ G_n))  \simeq \mathsf{RMod}^h_{G_n}(\mathsf{TAQ}(R),\Sigma^\infty_+ G_n).\]

By Proposition \ref{prop:taq of finite} and the finiteness of $G_i$, we can construct the following equivalences
\[ \mathsf{RMod}^h_{G_n}(\mathsf{TAQ}(R),\Sigma^\infty_+ G_n) \simeq (\mathsf{TAQ}(R)\wedge \Sigma^\infty_+ G_n)^{hG_n} \xrightarrow{\simeq} (\mathsf{TAQ}(R)\wedge D_{G_n} \wedge \Sigma^\infty_+ G_n )_{hG_n} \simeq  \mathsf{TAQ}(R)\wedge D_{G_n}. \]

We now address the assertion the derived (co)units are equivalences for level-finite right modules. For right modules homotopically concentrated in a single degree, this follows from Proposition \ref{prop: koszul dual of free} and Proposition \ref{prop: koszul dual of trivial}.

By induction on the truncation filtration, the claim holds for truncated right modules. However, the (co)unit of the $(K,K^{-1})$-adjunction at categorical degree $n$ depends only on the truncation at $n$. This can be seen by applying Proposition \ref{lem:taq concentrate} to the truncation filtration of mapping spectra. We conclude that the derived (co)unit is an equivalence for all level-finite modules.
\end{proof}

\begin{rem}
    In the special case $\mc{C}$ is the envelope of a level-finite operad, this demonstrates that Koszul duality is an equivalence for level-finite right modules, which appears to be an original result. This was conjectured in \cite[Theorem 9.1]{BehrensRezk}.
\end{rem}

Using the theory of Koszul duality, we can partially resolve the question of when $\mathsf{TAQ}$ detects equivalences.

\begin{cor}
  Let $\mc{C}$ be a dualizable $G_\ast$-category with $R,S$ a pair of level-finite right $\mc{C}$-modules. A map $R \rightarrow S$ is an equivalence, if and only if, the induced map $\mathsf{TAQ}(R) \rightarrow \mathsf{TAQ}(S)$ is an equivalence.
\end{cor}

\begin{proof}
    The functor 
    \[
    \mathsf{TAQ}:\mathsf{RMod}^\mathsf{fin}_\mc{C} \longrightarrow \mathsf{RMod}_{G_\ast},
    \]
    detects equivalences, if and only if, $\mathsf{TAQ}^\vee \wedge D_{G_\ast}$ detects weak equivalences. By Theorem \ref{thm:koszul duality for right modules} this last functor lifts to an equivalence of homotopy categories, 
    \[
    \mathsf{RMod}^\mathsf{fin}_\mc{C} \xrightarrow{\ \simeq \ }\mathsf{RMod}^\mathsf{fin,op}_{K(\mc{C})},
    \]
    which implies the result.
\end{proof}
Finally, we address to what extent Koszul duality is a duality for $G_\ast$-categories, i.e., under what conditions does one have $K(K(\mc{C})) \simeq \mc{C}$.

\begin{thm}\label{thm: double koszul dual}
   Let $\mc{C}$ be a dualizable $G_\ast$-category. The underlying $\mathsf{Spec}$-enriched $\infty$-categories of $\mathcal{C}$ and $K(K(\mathcal{C}))$ are equivalent over $G_\ast$. 
\end{thm}

\begin{proof}
    We have established that there is a spectral Quillen adjunction
    \[
\adjunction{K}{\mathsf{RMod}_\mc{C}}{\mathsf{RMod}_{K(\mc{C})}^\op}{K^{-1}}.
\]
This induces an adjunction on the corresponding spectrally enriched $\infty$-categories.

When restricted to the sub-$\infty$-categories of level-finite modules
\[
K : \mathsf{RMod}_\mc{C}^\mathsf{fin} \longrightarrow \mathsf{RMod}_{K(\mc{C})}^{\mathsf{fin}, \op}
\]
induces a categorical equivalence (in the sense of~\cite[Definition 5.5.2]{GepnerHaugseng}) of spectrally enriched $\infty$-categories by Theorem \ref{thm:koszul duality for right modules}. It follows from~\cite[Corollary 5.5.4]{GepnerHaugseng} that Koszul duality on level-finite right modules provides an equivalence of spectrally enriched $\infty$-categories. Consider the composite of functors between spectral $\infty$-categories
\[
\mc{C} \hookrightarrow \mathsf{RMod}_\mc{C}^\mathsf{fin} \xrightarrow{ \ K \ } \mathsf{RMod}_{K(\mc{C})}^{\mathsf{fin}, \op}, 
\]
of the Yoneda embedding $n \mapsto \mathsf{Free}_\mc{C}(\Sigma^\infty_+G_n)$ with Koszul duality. We claim that this diagram factors through the underlying spectral $\infty$-category of $K(K(\mc{C}))$. To see this, note that the spectral $\infty$-category $K(K(\mc{C}))$ is the full sub-$\infty$-category of $\mathsf{RMod}_{K(\mc{C})}^\op$ spanned by the objects $\mathsf{Triv}_{K(\mc{C})}(\Sigma^\infty_+ G_n)$. The required factorization follows from Proposition~\ref{prop: koszul dual of free}, since (on the derived level) there is an equivalence 
\[
K(\mathsf{Free}_\mc{C}(\Sigma^\infty_+G_n)) \simeq \mathsf{Triv}_{K(\mc{C})}(\Sigma^\infty_+ G_n). 
\]
It follows that there is a diagram 
\[\begin{tikzcd}
	{\mc{C}} & {\mathsf{RMod}^\mathsf{fin}_\mc{C}} & {\mathsf{RMod}_{K(\mc{C})}^{\mathsf{fin}, \op}} \\
	& {K(K(\mc{C}))}
	\arrow[hook, from=1-1, to=1-2]
	\arrow[dashed, from=1-1, to=2-2]
	\arrow["K", from=1-2, to=1-3]
	\arrow[hook, from=2-2, to=1-3]
\end{tikzcd}\]
of spectrally enriched $\infty$-categories over $G_\ast$. By~\cite[Corollary 5.3.8]{GepnerHaugseng} it suffices to show that the dotted arrow is fully faithful and essentially surjective on the level of spectral $\infty$-categories. The dotted arrow is fully faithful since all the functors induce equivalences on mapping spectra. By~\cite[Lemma 5.3.4]{GepnerHaugseng} to check essential surjectivity it suffices to show surjectivity on the set of objects, but this last is immediate from the definitions. 
\end{proof}

\section{Weiss calculus}

Weiss calculus is a homotopy theoretic tool developed to study functors from the category of Euclidean spaces to (pointed) spaces or spectra. It was originally developed by Weiss~\cite{WeissOrthogonal} with unstable applications in mind, but has found wide-reaching applications in stable homotopy theory. In this section we describe the construction of the Weiss tower, and provide several models for the Weiss derivatives for functors which take values in spectra.

\subsection{The Weiss tower}

Denote by $\mathsf{Vect}_\bb{R}$ the spectral category of finite-dimensional real inner product spaces\footnote{In the classic literature, this category (or at least the topologically enriched version) is denoted by $\es{J}$. We find this notation more informative, even if it does introduce some ambiguity about the morphisms in the category.} and linear isometric embeddings, with mapping spectra given by the pointed suspension spectrum of the space of linear isometric embeddings. As spectra, there are isomorphisms, 
\[
\mathsf{Vect}_\bb{R}(V, W) = \Sigma^\infty_+ O(W)/O(W-V),
\]
where $W -V$ denotes the orthogonal complement of $V$ in $W$ along some chosen (the choice is unimportant) linear isometric embedding. Given a functor $F: \mathsf{Vect}_\bb{R} \to \s$, Weiss calculus builds a tower of approximations 
\[\begin{tikzcd}
	&& F \\
	{P_\infty F} & \cdots & {P_nF} & \cdots & {P_1F} & {P_0F}
	\arrow[curve={height=6pt}, from=1-3, to=2-1]
	\arrow[from=1-3, to=2-3]
	\arrow[curve={height=-6pt}, from=1-3, to=2-5]
	\arrow[curve={height=-6pt}, from=1-3, to=2-6]
	\arrow[from=2-1, to=2-2]
	\arrow[from=2-2, to=2-3]
	\arrow[from=2-3, to=2-4]
	\arrow[from=2-4, to=2-5]
	\arrow[from=2-5, to=2-6]
\end{tikzcd}\]
under $F$, where $P_nF$ is the universal \textit{$n$-polynomial approximation} of $F$. A functor $F$ is $n$-polynomial if the canonical map
\[
F(V) \longrightarrow \underset{0 \neq U \subseteq \bb{R}^{n+1}}{\holim}F(V \oplus U) =:\tau_nF(V),
\]
is an equivalence for each $V \in \mathsf{Vect}_\bb{R}$. The object $P_nF(V)$ is defined to be the homotopy colimit of the diagram,
\[
F(V) \longrightarrow \tau_nF(V) \longrightarrow (\tau_n)^2F(V) \longrightarrow \cdots \longrightarrow (\tau_n)^kF(V) \longrightarrow \cdots.
\]
We will denote by $\poly{n}(\mathsf{Vect}_\bb{R}, \s)$ the (spectral) $\infty$-category of $n$-polynomial functors. Following Barnes and Oman~\cite[\S13]{BarnesOmanOrthogonal}, this may be modelled as a certain (stable) left Bousfield localization of the projective model structure on the category $\mathsf{Fun}(\mathsf{Vect}_\bb{R},\s)$.

\subsection{Weiss derivatives}
The $n$-th layer of the Weiss tower is the homotopy fiber
\[
D_nF = \hofibre(P_nF \longrightarrow P_{n-1}F).
\]

It measures the error between successive polynomial approximations and is an $n$-homogeneous functor, i.e., an $n$-polynomial functor with contractible $(n-1)$-polynomial approximation.  The key theorem in Weiss calculus is that $n$-homogeneous functors are completely classified by a spectrum with an action of $O(n)$, i.e., a Borel $O(n)$-spectrum. The spectrum classifying the $n$-th layer of the Weiss tower of $F$ is called the \emph{$n$-th derivative of $F$}. There are numerous constructions of this Borel $O(n)$-spectrum. In this subsection, we will provide a model following the original work of Weiss~\cite{WeissOrthogonal} and the model categorical formulation of Barnes and Oman~\cite{BarnesOmanOrthogonal}.

The following constructions are rather technical since we wish to work with $S$-modules rather than the more classical orthogonal spectra. Nevertheless, we include the details for completeness.  All of this may be formulated model independently using the language of $\infty$-categories, see ~\cite{Hendrian}. Finally, we wish to emphasize that this paper is largely about constructing more friendly models of these derivatives, and the reader might find the models of the next section more understandable, at least at first glance.

We first construct the derivative as an orthogonal spectrum with an action of $O(n)$, for details on orthogonal spectra we refer to ~\cite{MMSS}. Denote by $\mathsf{Sp}^\es{O}$ the category of orthogonal spectra. Denote by $\gamma_1(V,W)$ the orthogonal complement bundle over the space of linear isometries $\mathsf{Vect}_\bb{R}(V,W)$, with fiber over $f$ given by $f(V)^\perp$, and for $n\geq 1$, denote by $\gamma_n(V,W)$ the $n$-fold Whitney sum of $\gamma_1(V,W)$ with itself. For $n\geq 0$ define $\Sigma^\infty\mathsf{Vect}_{\bb{R},n}$ to be the spectral category with the same objects as $\mathsf{Vect}_\bb{R}$, but the spectrum of morphisms $\mathsf{Vect}_{\bb{R},n}(V,W)$ is given by the suspension spectrum of the Thom space of $\gamma_n(V,W)$\footnote{In the classical literature this category is denoted $\es{J}_n$.}. For each $n \geq 1$, the inclusion $\R^0 \subseteq \R^n$ induces an inclusion of topological categories
\[
\iota_n: \mathsf{Vect}_\bb{R} \hookrightarrow \Sigma^\infty\mathsf{Vect}_{\bb{R},n},
\]
and hence an adjunction
\[
\adjunction{\res_0^n/O(n)}{\Fun_{O(n)}(\Sigma^\infty\mathsf{Vect}_{\bb{R},n}, (\mathsf{Sp}^\es{O})^{BO(n)})}{\Fun(\mathsf{Vect}_{\bb{R}}, \mathsf{Sp}^\es{O})}{\ind_0^n\varepsilon^\ast},
\]
see ~\cite[\S3]{WeissOrthogonal} or \cite[\S11]{BarnesOmanOrthogonal}. Restricting to $n$-homogeneous objects on the right by viewing the $n$-homogeneous functors as bifibrant objects in a certain model structure, see ~\cite[\S 6]{BarnesOmanOrthogonal}, we obtain a Quillen equivalence
\[
\adjunction{\res_0^n/O(n)}{\Fun_{O(n)}(\Sigma^\infty\mathsf{Vect}_{\bb{R},n}, (\mathsf{Sp}^\es{O})^{BO(n)})}{\homog{n}(\mathsf{Vect}_{\bb{R}}, \mathsf{Sp}^\es{O})}{\ind_0^n\varepsilon^\ast},
\]
where the left-hand side is equipped with the $n$-stable model structure in the sense of~\cite[Proposition 7.14]{BarnesOmanOrthogonal}, see also~\cite[Theorem 10.1]{BarnesOmanOrthogonal}. Following~\cite[Definition 8.2]{BarnesOmanOrthogonal}, there is a functor of spectral categories $\alpha_n: \mathsf{Vect}_{\bb{R},n} \to \mathsf{Vect}_{\bb{R},1}$, given on objects by $\alpha_n(V) = \bb{R}^n \otimes V$. This induces an adjoint pair
\[
\adjunction{(\alpha_n)_!}{\Fun_{O(n)}(\mathsf{Vect}_{\bb{R},n}, (\mathsf{Sp}^\es{O})^{BO(n)})}{\Fun(\mathsf{Vect}_{\bb{R},1}, (\mathsf{Sp}^\es{O})^{BO(n)})}{(\alpha_n)^\ast},
\]
which is a Quillen equivalence by~\cite[\S8, \S11]{BarnesOmanOrthogonal}. The right-hand category is canonically equivalent to the category of orthogonal bispectra with an $O(n)$-action. Via the diagonal functor
\[
d:\mathsf{Sp}^\es{O}(\mathsf{Sp}^\es{O})^{BO(n)} \to (\mathsf{Sp}^\es{O})^{BO(n)},
\]
there is a Quillen equivalence between bispectra with an action of $O(n)$ and spectra with an action of $O(n)$. 

For a functor $F: \mathsf{Vect}_\bb{R} \to \mathsf{Sp}^\es{O}$, we define the $n$-th derivative of $F$ to be the derived image of $F$ under this zigzag of Quillen equivalences. To obtain a model for functors from the category of Euclidean spaces to $S$-modules, we use the Quillen equivalence
\[
\adjunction{\bb{N}^\sharp}{\s}{\mathsf{Sp}^\es{O}}{\bb{N}},
\]
of Mandell-May~\cite{MandellMay}, i.e., Given a functor $F: \mathsf{Vect}_\bb{R} \to \s$, postcomposition with the strong symmetric monoidal left Quillen equivalence
\[
\bb{N}^\# : \s \longrightarrow \mathsf{Sp}^\es{O},
\]
induced a functor $\bb{N}^\sharp \circ F: \mathsf{Vect}_\bb{R} \to \mathsf{Sp}^\es{O}$. We can then apply the above construction to obtain an orthogonal spectrum with an $O(n)$-action, and then employ the lax symmetric monoidal right Quillen functor
\[
\bb{N} : \mathsf{Sp}^\es{O} \longrightarrow \s,
\]
to obtain an $S$-module with an action of $O(n)$. 

\begin{definition}
Let $F: \mathsf{Vect}_\bb{R} \to \s$. Define the $n$-th Weiss derivative of $F$, denoted $\partial_n^WF$, to be the derived image of $D_nF$ in spectra with an action of $O(n)$, under the above equivalences of categories:
\[
\partial_n^W F = (\bb{N})\circ (d \circ (\alpha_n)_! \circ \ind_0^n\varepsilon^\ast)(D_n( \bb{N}^\# \circ F)),
\]
where to ease notation we left implicit the simplicial (co)fibrant replacements.
\end{definition}

One of the most important theorems of Weiss calculus is the classification of homogeneous functors ~\cite[Theorem 7.3]{WeissOrthogonal}.

\begin{prop}
Let $F: \mathsf{Vect}_\bb{R} \to \s$. For each $n \in \bb{N}$, there is an equivalence of functors
\[
D_nF \simeq (\partial_nF \wedge S^{n(-)})_{hO(n)},
\]
where $S^{n(-)}$ is the functor given by sending $V$ to the suspension spectrum of the one-point compactification of $\bb{R}^n \otimes V$. 
\end{prop}

An important example for us is the derivatives of representables, as showing that our various models for Weiss derivatives agree on representables is usually the easiest way to show they agree in general.

\begin{ex}\label{ex: Weiss derivative of rep}
The $n$-th Weiss derivative of the representable functor $\mathsf{Vect}_\bb{R}(V,-)$ is equivalent to $\Omega^{nV}(D_{O(n)} \wedge \mathsf{Vect}_\bb{R}(\bb{R}^n, V))$, i.e., the Spanier--Whitehead dual of the spectrum  
\[
\s(D_{O(n)} \wedge  \mathsf{Vect}_\bb{R}(\R^n, V), S^{nV}).
\]

\end{ex}
\begin{proof}
This computation is immediate from Miller's stable splitting of Stiefel manifolds~\cite{MillerSplitting}, and the definition of $\partial_nF$. Details can be readily extracted from Arone's~\cite{AroneBOBU} computations of the Weiss tower of $\sf{BO}(-)$, or the exposition of Arone in~\cite{AroneMitchellRichter}. 
\end{proof}

\begin{rem}
In many ways, the models we provide in the rest of this article and more elementary than the construction provided by Weiss~\cite{WeissOrthogonal} and Barnes and Oman~\cite{BarnesOmanOrthogonal}. This is not surprising, a similar phenomenon happens in Goodwillie calculus. In fact, we will see later that the fake Weiss tower we introduce in Section~\ref{subsection: Koszul duality model} agrees with the Weiss tower for representable functors, and so by Kan extending yields a fairly simple model for the Weiss tower of a general functor. 
\end{rem}

\subsection{A Spanier--Whitehead duality model for Weiss derivatives} \label{subsection:spanier whitehead model}
We now introduce a new model for Weiss derivatives. This construction is Koszul dual to the model of Section~\ref{subsection: Koszul duality model}.

For a finite-dimensional real inner product space $V$, we denote by $R_V: \mathsf{Vect}_\bb{R} \to \s$ the representable functor
\[
W \longmapsto \mathsf{Vect}_\bb{R}(V,W).
\]
Define the \emph{linear fat diagonal} to be the functor $\mathsf{DI}_n : \mathsf{Vect}_\bb{R} \to \s$, by defining $\sf{DI}_n(V)$ to be the one-point compactification of the space of non-injective linear maps $\mathbb{R}^n \to V$. Note that by the identification of $\bb{R}^n \otimes U$ with $\Hom(\bb{R}^n, U)$ the linear fat diagonal is homeomorphic to $\Hom(\bb{R}^n, U)\setminus \mathsf{Vect}_\bb{R}(\bb{R}^n, U)$ the complement of the space of linear isometries in the space of all linear homomorphisms.

\begin{definition}
For a representable functor $R_V: \mathsf{Vect}_\bb{R} \to \s$, define the \emph{$n$-th Spanier--Whitehead derivative} $\partial_n^{SW}(R_V)$ to be the Spanier--Whitehead dual of the spectrum of natural transformations $\nat(R_V \wedge D_{O(n)}, S^{n(-)}/\sf{DI}_n(-))$, i.e., 
\[
\partial_n^{SW}(R_V) = (\nat(R_V \wedge D_{O(n)}, S^{n(-)}/\sf{DI}_n(-)))^\vee.
\]

\end{definition}

Notice that the Spanier--Whitehead derivatives defined a functor
\[
\partial_\ast^{SW} R_{(-)} : \mathsf{Vect}_\bb{R}^\op \longrightarrow \mathsf{RMod}_{O(\ast)}, V \longmapsto \partial_\ast(R_V),
\]
where $\mathsf{RMod}_{O(\ast)}$ is the category of orthogonal sequences, i.e., right $O(\ast)$-modules, cf.~Example~\ref{ex: augmented cats}.

\begin{ex}\label{ex: derivative of rep}
The $n$-th Spanier--Whitehead derivative of the representable functor $R_V$ is equivalent to $\Omega^{nV}(\mathsf{Vect}_\bb{R}(\bb{R}^n, V) \wedge D_{O(n)})$, i.e., to the Spanier--Whitehead dual of the spectrum
\[
\s(D_{O(n)} \wedge  \mathsf{Vect}_\bb{R}(\R^n, V), S^{nV}).
\]
\end{ex}
\begin{proof}
By the Yoneda Lemma we have an identification
\[
\nat(R_V \wedge D_{O(n)}, S^{n(-)}/\sf{DI}_n(-)) \cong \s(D_{O(n)}, S^{nV}/\mathsf{DI}_n(V)),
\]
and the result follows from a combination of Atiyah duality which identifies $S^{nV}/\mathsf{DI}_n(V)$ with $\s(\mathsf{Vect}_\bb{R}(\R^n, V), S^{nV})$ and Spanier--Whitehead duality.
\end{proof}

To define Spanier--Whitehead derivatives for arbitrary functors, we will left Kan extend the functor
\[
\partial_\ast^{SW} R_{(-)} : \mathsf{Vect}_{\bb{R}}^\op \longrightarrow \mathsf{RMod}_{O(\ast)},
\]
along the (contravariant) Yoneda embedding
\[
\mathsf{Vect}_{\bb{R}}^\op \hookrightarrow \Fun(\mathsf{Vect}_\bb{R}, \s).
\]

\begin{definition}\label{definition: spanier whitehead derivative}
Let $F: \mathsf{Vect}_\bb{R} \to \s$. Define the \emph{$n$-th Spanier--Whitehead derivative} of $F$ as
\[
\partial_n^{SW} F = \int^{V \in \mathsf{Vect}_\bb{R}}~c(\partial_n(R_V)) \wedge F(V),
\]
where $c$ is a simplicial cofibrant replacement in the category of orthogonal sequences.
\end{definition}

\begin{rem}
For $F: \mathsf{Vect}_\bb{R} \to \s$, a finite cell complex in the category of orthogonal functors, there is an equivalence 
\[
\partial_n^{SW} F \simeq \nat(F \wedge D_{O(n)}, S^{n(-)}/\mathsf{DI}_n(-))^\vee,
\]
between the Spanier--Whitehead derivatives of $F$ and the Spanier--Whitehead dual of the spectrum of natural transformations from $F \wedge D_{O(n)}$ to $S^{n(-)}/\mathsf{DI}_n(-)$. This identification follows from the case of representable functors combined with the fact that finite cell complexes in the category of orthogonal functors may be constructed by a finite sequence of extensions by representables.
\end{rem}

Up to homotopy, the Spanier--Whitehead derivatives agree with the Weiss derivatives. The following argument is similar to that of Arone and Ching~\cite[Lemma 4.3]{ACClassification}, which is an extension of an argument of Oman~\cite{Oman}.

\begin{lem}\label{lem: SW and OG agree}
For a cofibrant functor $F: \mathsf{Vect}_\bb{R} \to \s$, the $n$-th Spanier--Whitehead derivative of $F$ is weakly equivalent to the $n$-th Weiss derivative of $F$.
\end{lem}
\begin{proof}
We wish to construct a (zigzag) of maps 
\[
\partial_n^{SW}(F) \longrightarrow \partial_n^W(F),
\]
which we will demonstrate is an equivalence. Since functors are functorially the coends of representables, it suffices to construct the zigzag for functors of the form $R_V$ and demonstrate that both constructions preserve homotopy colimits and equivalences between cofibrant objects. By Example~\ref{ex: Weiss derivative of rep} and Example~\ref{ex: derivative of rep}, the claimed maps exist for representable functors, hence it suffices to show that $\partial_n^{SW}$ and $\partial_n^W$ both preserve homotopy colimits and equivalences. For homotopy colimits, note that $\partial_n^{SW}$ is defined as a colimit, and these commute on cofibrant functors. Similarly, it is clear that $\partial_n^{SW}$ preserves equivalences between cofibrant objects. For $\partial_n^W$, note that it is the composite of left and right Quillen equivalences (together with appropriate simplicial (co)fibrant replacements) and hence commutes with homotopy colimits and preserves equivalences between cofibrant objects.
\end{proof}

We conclude this section by showing that the Spanier--Whitehead derivatives of a cofibrant functor are cofibrant as a consequence of the following Quillen adjunction.

\begin{lem}\label{lem: SW derivative a left adjoint}
There is an adjoint pair
\[
\adjunction{\partial_\ast}{\Fun(\mathsf{Vect}_\bb{R}, \s)}{\sf{RMod}_{O(\ast)}}{\Phi},
\]
which is a simplicial Quillen adjunction.
\end{lem}
\begin{proof}
Let $R$ be a right $O(\ast)$-module. The right adjoint $\Phi$ is defined as
\begin{align*}
\Phi(S)(V) &= \sf{RMod}_{O(\ast)}(\partial_\ast(R_V), S) \\ &=\int_{n \in \mathbb{N}}\s(c(\partial_n(\Sigma^\infty R_V)), S(n)) \\ &= \prod_{n \in {O(\ast)}}\s(c(\partial_n(\Sigma^\infty R_V)), S(n))^{O(n)}. 
\end{align*}
Via a standard ``calculus of (co)ends'' argument, one can readily see that $\Phi$ is right adjoint to $\partial_\ast$. To see that the adjoint pair is Quillen simply notice that the right adjoint preserves fibrations and acyclic fibrations as these are defined levelwise, and we picked a cofibrant model for the derivatives of representables.
\end{proof}

\section{Koszul duality and Weiss calculus}

In Section \ref{section:koszul duality}, we investigated Koszul duality for $A_\ast$-categories and their right modules. In particular, we observed that Koszul duality can be interpreted as a Fourier transform with respect to the characters given by the trivial right modules. We demonstrated for dualizable $G_\ast$-categories (Definition \ref{def: dualizable category}) like $\mathsf{OEpi}$ that Koszul duality determined an equivalence of level-finite right module categories.

For a general category $C$ equipped with a subcategory of $\mathsf{Fun}(C,\s)$ of characters, one could write down a similar transform. In the context of Weiss calculus, the obvious characters are the homogeneous functors $S^{n(-)}$. It was originally conjectured by Behrens \cite{behrensOrthNote} that such transforms, which take values in right modules over the category $\mathsf{OEpi}$, might play an important role in Weiss calculus. Indeed, the work of Arone--Ching \cite{ACOperads} implicitly takes this perspective when dealing with endofunctors of spectra.

In this section, we investigate the Fourier transform $\partial^\ast$ of functors $\mathsf{Vect}_\bb{R} \rightarrow \mathsf{Spec}$. In particular, we validate Behrens' conjecture by computing the relation between Fourier transforms of orthogonal functors, Koszul duality, and Weiss calculus. The story is summarized by the diagram of adjunctions: 
\[\begin{tikzcd}
	& {\Fun(\mathsf{Vect}_\bb{R}, \s)} \\
	{\mathsf{RMod}_{K(\mathsf{OEpi})}} && {\mathsf{RMod}_{\mathsf{OEpi}}^\mathsf{op}}
	\arrow["{\partial_\ast}"', shift right, from=1-2, to=2-1]
	\arrow["{\partial^\ast}", shift left, from=1-2, to=2-3]
	\arrow["{\Theta^K}"', shift right, from=2-1, to=1-2]
	\arrow["{K^{-1}}"', shift right, from=2-1, to=2-3]
	\arrow["\Theta", shift left, from=2-3, to=1-2]
	\arrow["K"', shift right, from=2-3, to=2-1]
\end{tikzcd}\]

With suitable finiteness conditions, the left adjoints commute and the right adjoints commute, and so the Koszul dual of the Fourier transform of an orthogonal functor yields a model for its Weiss derivatives. Unlike Koszul duality for right modules, we find that the orthogonal Fourier transform is far from an equivalence, even on polynomial functors.

We show that the obstructions for the composite of the orthogonal Fourier transform with the inverse orthogonal Fourier transform \[\theta \circ \partial^\ast = \theta^K \circ \partial_\ast =:P^\mathsf{fake}_\infty(-)\] to agree with $P_\infty(-)$ are precisely the $O(\ast)$-Tate spectra of the Weiss derivatives.

 These results are analogous to what happens in the case of functors $\T \rightarrow \s$. There the derivatives are known to form right modules over the Koszul dual of the category of finite sets and surjections, which is also known as the envelope of the Lie operad \cite{ACOperads}. Further, the obstruction for the analogous fake Goodwillie tower to agree with the Goodwillie tower lie in the $\Sigma_\ast$-Tate spectra~\cite[Remark 4.2.27]{ACOperads}.

\subsection{Koszul dual derivatives}
We now formally introduce the category $\mathsf{OEpi}$ of orthogonal epimorphisms and verify that it fits into our framework of Koszul duality. Recall that $\mathsf{OEpi}$ denotes the category of orthogonal epimorphisms and is defined as the opposite of the category of finite-dimensional inner product spaces with linear isometric embeddings. 

\begin{prop}\label{prop: OEpi dualizable}
    The category $\mathsf{OEpi}$ is a dualizable $O(\ast)$-category.
\end{prop}
\begin{proof}
    The endomorphisms of $\mathsf{OEpi}$ are given by $\Sigma^\infty_+ O(n)$, hence are suspension spectra of finite CW complexes. The spectrum $\Sigma^\infty_+ O(0)$ is $S^0$ since there is a single automorphism of $\mathbb{R}^0$. Since linear injections cannot decrease dimension, the order requirement on mapping spectra holds.
   The dualizing spectra of $O(n)$ are spherical by Theorem \ref{thm:norm for g spectra} since they are compact Lie groups.  The action of $O(n)$ on the right of $\mathsf{OEpi}(m,n)$ is a finite $O(n)$-spectrum since linear injections are have left inverses.
\end{proof}

As a consequence of Theorem \ref{thm:koszul duality for right modules}, Koszul duality is well-behaved for $\mathsf{OEpi}$ and its level-finite right modules. One can produce a wealth of interesting right $\mathsf{OEpi}$-modules in the following way:

\begin{definition}
The \emph{Koszul dual} derivatives $\partial^\ast F$ of a functor $F: \mathsf{Vect}_\bb{R} \to \s$ are the right $\mathsf{OEpi}$-module
\[
\partial^\ast F := \nat(F, S^{\ast(-)}).
\]
The right module structure is given by composition with the natural transformations 
\[
S^{n(-)} \longrightarrow S^{m(-)}
\]
given by tensoring with $i \in\mathsf{Vect}_\bb{R}(\bb{R}^m,\bb{R}^n)$. 
\end{definition}

We may think of the Koszul dual derivative $\partial^\ast F$ as the Fourier transform of $F$ with respect to the characters $S^{\ast(-)}$. As before, there is a dual Fourier transform given by
\[
\Theta: \mathsf{RMod}_{\sf{OEpi}}^\op \longrightarrow \Fun(\mathsf{Vect}_\bb{R}, \s)\]\[R \longmapsto \mathsf{RMod}_{\sf{OEpi}}(R, S^{\ast(-)}).
\]

\begin{prop}\label{prop: koszul dual derivatives adjunction}
The Koszul dual derivatives are part of a spectrally enriched Quillen adjunction
    \[
\adjunction{\partial^\ast}{\mathsf{Fun}(\mathsf{Vect}_\bb{R},\s)}{\mathsf{RMod}_{\mathsf{OEpi}}^{\mathsf{op}}}{\Theta}.
\]
\end{prop}
\begin{proof}
A calculus of (co)ends argument provides a natural isomorphism
\[
\mathsf{RMod}_{\mathsf{OEpi}}^\op(\partial^\ast F, R) = \mathsf{RMod}_\mathsf{OEpi}(R, \partial^\ast F) \cong \nat(F, \mathsf{RMod}_\mathsf{OEpi}(R, S^{\ast(-)})),
\]
for $F: \mathsf{Vect}_\bb{R} \to \s$ and $R \in \mathsf{RMod}_\mathsf{OEpi}$, and hence the functors in question are adjoint. The Quillen property follows since the right adjoint preserves (acyclic) fibrations by the axioms of an enriched model category.
\end{proof}

 Typically, we would restrain ourselves to taking the Koszul dual derivatives of cofibrant functors. However, there are cases when it is not necessary to pass to a cofibrant replacement.\footnote{This is similar to how we restrict to pseudo-cofibrant mapping spectra rather than strictly cofibrant mapping spectra in Definition~\ref{def: A* category}.}

\begin{definition}\label{def: correct natutral transformation spectra}
We will say that a functor $F: \mathsf{Vect}_\bb{R} \to \s$ has the \textit{correct natural transformation spectra}, if precomposition with a cofibrant replacement induces an equivalence of natural transformation spectra.
\end{definition}

\begin{lem} \label{lem: correct natural transformations for rep}
The representable functor $R_V=\mathsf{Vect}_\bb{R}(V,-)$ has the correct natural transformation spectra.
\end{lem}
\begin{proof}
A cofibrant replacement of $R_V$ is given as $R_V \wedge \bb{S}_c$. The derived Yoneda Lemma shows that $\mathsf{nat}(R_V \wedge \bb{S}_c, F) \cong \mathsf{Spec}(\bb{S}_c,F(V))$ which is equivalent to $F(V) \cong \mathsf{nat}(R_V , F)$, as desired.
\end{proof}

By the Yoneda lemma, we have the following identification of the Koszul dual of representable functors.

\begin{ex}\label{ex: koszul dual derivative of reps}
There is an isomorphism of right $\mathsf{OEpi}$-modules:
\[\partial^\ast \mathsf{Vect}_\bb{R}(V,-) \cong S^{\ast V}\]
where the action of $\mathsf{OEpi}$ given by smashing with $i \in \mathsf{Vect}_\bb{R}(\bb{R}^m,\bb{R}^n)$.
\end{ex}

\begin{lem}[{\cite[Example 4.1]{ReisWeiss}}]\label{lem: factorization}
Let $F$ and $G$ be orthogonal functors. If $G$ is $n$-polynomial the map
 \[
 \nat^\sf{h}(P_n(F),G) \longrightarrow \nat^\sf{h}(F,G),
 \]
  induced by the universal map $F \to P_nF$ is an equivalence.
\end{lem}

\begin{cor} \label{cor: equivalence of Koszul dual derivatives}
    If $F \rightarrow G$ is a natural transformation of functors with the correct natural transformation spectra which induces equivalences on $\partial_i(-)$ for $i \leq j$, then the map
    \[(\partial^\ast G )^{\leq j} \rightarrow (\partial^\ast F)^{\leq j}  \]
    is an equivalence of right $\mathsf{OEpi}$-modules.
\end{cor}

One avatar of duality in functor calculus is the stark difference in complexity between computing the derivatives of homogeneous functors versus computing the derivatives of representable functors. The first is essentially by definition, while the latter is in general tricky and tends to require geometric insight. For the Koszul dual derivatives, this pattern is reversed. The Koszul dual derivatives of representable functors are easily computed, while the computation for homogeneous functors is more difficult.  

Given a Borel $O(n)$-spectrum $X$, let $H_n(X)$ denote a cofibrant replacement of the homogeneous functor $(X \wedge S^{n(-)})_{hO(n)}$.
\begin{prop}\label{prop:koszul derivative homogeneous}
    If $X$ is an $O(n)$-spectrum for which the underlying spectrum is finite, there is an equivalence of right $\mathsf{OEpi}$-modules \[\partial^\ast(H_n(X)) \simeq \mathsf{Free}_{\mathsf{OEpi}}(X^\vee \wedge D_{O(n)}).\]
\end{prop}

\begin{proof}
     The natural transformations between homogeneous functors were computed \cite[Theorem 3.2]{Arone_Dwyer_Lesh_2008} to be     
\[\nat^\sf{h}(H_n(X),H_m(Y)) \simeq \mathsf{Spec}^h(X, (Y \wedge \mathsf{Vect}_\bb{R}(\bb{R}^n, \bb{R}^m))_{hO(m)} )^{hO(n)}\]
In our case, $Y= \Sigma^\infty_+ O(m)$ and this reduces to
\[\mathsf{Spec}(X, \mathsf{Vect}_\bb{R}(\bb{R}^n,\bb{R}^m) )^{hO(n)}.\]
By the finiteness of $X$, we can write this as
\[
(X^\vee \wedge\mathsf{Vect}_\bb{R}(\bb{R}^n,\bb{R}^m)_{hO(n)}.
\]
The action of $O(n)$ on $\mathsf{Vect}_\bb{R}(\bb{R}^n,\bb{R}^m)$ is free, and so the norm map is an equivalence:
\[\nat^\sf{h}(H_n(X),H_m(Y)) \simeq (X^\vee \wedge D_{O(n)} \wedge \mathsf{Vect}_\bb{R}(\bb{R}^n,\bb{R}^m))_{hO(n)}. \]
This is the formula for the free right module on $X^\vee \wedge D_{O(n)}$.
\end{proof}

\subsection{The fake Weiss tower and Koszul dual derivatives}\label{subsection: fake weiss tower finite}
In this section, we give a construction of the homotopy type of the ``fake Weiss tower''. This tower is ``fake'' in the sense that their layers appear to be those of the Weiss tower, but the homotopy orbits are replaced by homotopy fixed points. This construction will be valid for functors with level-finite derivatives. The construction also makes sense for functors with infinite derivatives (by extending by pro-objects), but its connection to Weiss calculus is less direct. We emphasize that these two constructions will agree up to homotopy for functors with level-finite derivatives, and so we will refer to them both as the fake Weiss approximations unless confusion could occur.

\begin{definition}
    For a functor $F : \mathsf{Vect}_\bb{R} \to \s$ with level-finite derivatives,  the \emph{fake Weiss approximation}
    \[F \rightarrow P^\infty_\mathsf{fake}(F):= \Theta \circ \partial^\ast\]
    is the derived unit of the $(\partial^\ast,\Theta)$ adjunction.
\end{definition}

Explicitly, the map
\[F(V) \rightarrow P^\infty_\mathsf{fake}(F)(V)\]
is given by
\[F(V) \cong \nat(R_V,F) \xrightarrow{\partial^\ast} \mathsf{RMod}^h_{\mathsf{OEpi}}(\partial^\ast F, \partial^\ast R_V)\simeq \mathsf{RMod}^h_{\mathsf{OEpi}}(\partial^\ast F, S^{\ast V}),\]
and is covariantly functorial in $F$.

\begin{definition}\label{def: fake weiss finite derivative} 
The \emph{fake Weiss tower} of a functor  $F: \mathsf{Vect}_\bb{R} \to \s$ with level-finite derivatives is the truncation tower for $\mathsf{RMod}^h_{\mathsf{OEpi}}(\partial^\ast F, \partial^\ast R_V)$:
\[\begin{tikzcd}
	&& {F(V)} \\
	{P^\infty_\mathsf{fake} F(V)} & \cdots & {P^n_\mathsf{fake}F(V)} & \cdots & {P^1_\mathsf{fake}F(V)} & {P^0_\mathsf{fake}F(V)}
	\arrow[curve={height=6pt}, from=1-3, to=2-1]
	\arrow[from=1-3, to=2-3]
	\arrow[curve={height=-6pt}, from=1-3, to=2-5]
	\arrow[curve={height=-6pt}, from=1-3, to=2-6]
	\arrow[from=2-1, to=2-2]
	\arrow[from=2-2, to=2-3]
	\arrow[from=2-3, to=2-4]
	\arrow[from=2-4, to=2-5]
	\arrow[from=2-5, to=2-6]
\end{tikzcd}\]
    
\end{definition}

The notation $P^n_\mathsf{fake}(F)$ is used to distinguish this construction from a later construction $P_n^\mathsf{fake}(F)$ of the fake Weiss approximations of $F$ in Section \ref{subsection:enriched model of fake Weiss tower} which exists in more generality, though these agree up to equivalence for functors with level-finite derivatives. We will now investigate how the fake Weiss tower relates to the Weiss tower. 

\begin{definition}
    Let $F : \mathsf{Vect}_\bb{R} \to \s$ be a functor with level-finite derivatives. The spectrum $D^n_\mathsf{fake}(F)(V)$ is 
    \[D^n_\mathsf{fake}(F)(V)\coloneq\mathsf{hofiber}(P^n_\mathsf{fake}(F)(V) \rightarrow P^{n-1}_\mathsf{fake}(F)(V)).\]
\end{definition}

\begin{prop} \label{prop:fake transformation}
    Let $F : \mathsf{Vect}_\bb{R} \to \s$ be a functor with level-finite derivatives. The natural transformation $F \rightarrow P^n_\mathsf{fake}(F)$ factors up to homotopy through $P_n(F)$.
\end{prop}
\begin{proof}
    This follows from Corollary \ref{cor: equivalence of Koszul dual derivatives} and the definition of the fake approximations in terms of natural transformations.
\end{proof}

\begin{prop} \label{prop:fibers real to fake}
    Let $F : \mathsf{Vect}_\bb{R} \to \s$ be a functor with level-finite derivatives. There is a homotopy commuting square

\begin{center}
\begin{tikzcd}
P_i(F)(V) \arrow[d] \arrow[r] & P^i_{\mathsf{fake}}(F)(V) \arrow[d] \\
P_{i-1}(F)(V) \arrow[r]       & P^{i-1}_{\mathsf{fake}}(F)(V)      
\end{tikzcd}
\end{center}
which induces a map on the layers of the form 
\begin{center}
\begin{tikzcd}
(\partial_nF \wedge S^{nV})_{hO(n)}  \arrow[r] & (K(\partial^\ast F)(n)  \wedge D^\vee_{O(n)} \wedge  S^{nV})^{hO(n)} \arrow[d,"\simeq"] \\
D_n(F)(V) \arrow[u, "\simeq"]                            & D^n_\mathsf{fake}(F)(V)                                                                  
\end{tikzcd}
\end{center}
up to homotopy.
\end{prop}
\begin{proof}
    The square follows from Corollary \ref{cor: equivalence of Koszul dual derivatives}. The calculation of the layers of the fake tower is the combination of Proposition \ref{prop: layers of truncation tower} and Theorem \ref{thm:koszul duality for right modules}, so long as we know that $\partial^\ast F$ is level-finite. For homogeneous $F$, this is a consequence of Proposition \ref{prop:koszul derivative homogeneous}, and for general $F$ it is achieved by induction after appealing to Lemma \ref{lem: factorization}.
\end{proof}

At this point, it should become more obvious why we have chosen to name $\partial^\ast F$ the Koszul dual derivatives. Ultimately, we wish to show that as orthogonal sequences $K(\partial^\ast F) \simeq \partial_\ast (F)$, and that the map 
\[D_n(F)(V)\rightarrow D^n_\mathsf{fake}(F)(V)\]
is the dual norm map.

\begin{cor}\label{cor: fake for homogeneous}
If $X$ is a Borel $O(n)$-spectrum whose underlying spectrum is finite, then
\[
D^m_\mathsf{fake}(H_n(X))(V) \simeq
\begin{cases}
    (X \wedge  D^\vee_{O(n)} \wedge S^{nV})^{hO(n)} & \text{if $m=n$} \\
    \ast & \text{if $m \neq n$}.
\end{cases}
\]
\end{cor}
\begin{proof}
     This follows from explicit computation of the layers given in Proposition \ref{prop:fibers real to fake} using  Proposition \ref{prop:koszul derivative homogeneous} and Proposition \ref{prop: koszul dual of free}.
\end{proof}

The following proposition records the formal properties of the fake Weiss tower. Given a functor 
\[
D : \Fun(\mathsf{Vect}_\bb{R}, \s) \longrightarrow  \mc{C},
\]
for some spectral model category $\mc{C}$, we define a natural transformation $F \to G$ to be a $D$-equivalence if $D(F) \to D(G)$ is a weak equivalence in $\mc{C}$. 

\begin{prop}\label{prop: formal properties of fake tower} \label{prop: Di equivalence and fake layers}
For each $n \in \bb{N}$, the $n$-th fake Weiss approximation $P^n_\mathsf{fake}$ and the $n$-th layer $D^n_\mathsf{fake}$ of the fake Weiss tower preserve
\begin{enumerate}
    \item homotopy (co)fiber sequences of functors; and,
    \item $\partial_{\leq n}$-equivalences.
\end{enumerate}
In addition, the $n$-th layer of $D^n_\mathsf{fake}$ of the fake Weiss tower preserves $D_n$-equivalences As a consequence, there is a natural zigzag 
  \[
 D^n_\mathsf{fake}(H_n(\partial_n(F)) \xrightarrow {\ \simeq \ } D^n_\mathsf{fake}(P_n(F))  \xleftarrow{\ \simeq \ } D^n_\mathsf{fake} (F),
 \] 
of weak equivalences.
\end{prop}
\begin{proof}
 The first fact boils down to the fact that taking natural transformations preserves (co)fiber sequences. The second fact follows from Proposition \ref{cor: equivalence of Koszul dual derivatives}. The most interesting is the third claim.

    Consider the map of fiber sequences 

    \begin{center}
\begin{tikzcd}
D_n(F) \arrow[d, "\simeq"] \arrow[r] & P_n(F) \arrow[d] \arrow[r] & P_{n-1}(F) \arrow[d] \\
D_n(G) \arrow[r]                     & P_n(G) \arrow[r]           & P_{n-1}(G)          
\end{tikzcd}
    \end{center}
By $(1)$ applied to the $n$-th layer of the fake Weiss tower we have another map of fiber sequences
\begin{center}
\begin{tikzcd}
D^n_\mathsf{fake}(D_n(F)) \arrow[d, "\simeq"] \arrow[r] & D^n_\mathsf{fake}(P_n(F)) \arrow[d] \arrow[r] & D^n_\mathsf{fake}(P_{n-1}(F)) \arrow[d] \\
D^n_\mathsf{fake}(D_n(G)) \arrow[r]                     & D^n_\mathsf{fake}(P_n(G)) \arrow[r]           & D^n_\mathsf{fake}(P_{n-1}(G))          
\end{tikzcd}
    \end{center}
    By induction over polynomial functors, the right column is contractible as long as \[D^n_\mathsf{fake}(H_{n-1}(X)) \simeq \ast, \] and so we would conclude the result for $n$-polynomial functors. This is the content of Corollary \ref{cor: fake for homogeneous}.
    
    The general case follows from $(2)$ since $F \rightarrow P_i(F)$ is an equivalence on $\partial_j$ for $j \leq i$.
\end{proof}

We now state the precise relationship between the derivatives and Koszul dual derivatives, along with the relationship of $D_n(F)$ and $D^n_\mathsf{fake}$.

\begin{thm}\label{thm: norm map finite derivatives}
Let $F: \mathsf{Vect}_\bb{R} \to \s$ with level-finite derivatives. There is an equivalence
    \[\partial_\ast (F) \simeq K(\partial^\ast (F)),\]
and the map
    \begin{center}
\begin{tikzcd}
(\partial_nF \wedge S^{nV})_{hO(n)}  \arrow[r, "\simeq"] & (K(\partial^\ast F)(n)  \wedge D^\vee_{O(n)} \wedge  S^{nV})^{hO(n)} \arrow[d] \\
D_n(F)(V) \arrow[u, "\simeq"]                            & D^n_\mathsf{fake}(F)(V)                                                                  
\end{tikzcd}
\end{center}
    is equivalent to the dual norm
    \[(\partial_n(F) \wedge S^{nV})_{hO(n)} \rightarrow (\partial_n(F) \wedge D_{O(n)}^\vee \wedge S^{nV})^{hO(n)} .\]
  
\end{thm}
\begin{proof}
   The first fact follows from the description of the layers of the fake tower in Proposition \ref{prop:fibers real to fake}, the computation for homogeneous functors in Corollary \ref{cor: fake for homogeneous}, and the $D_i$-invariance of Proposition \ref{prop: formal properties of fake tower}.
   
   To identify the map of layers as the norm map, we construct a natural transformation in the homotopy category of functors $\mathsf{RMod}_{O(n)} \rightarrow \mathsf{Spec}$ of the form
    \[(-)_{hO(n)} \rightarrow (- \wedge D_{O(n)}^\vee)^{hO(n)}\]
    by evaluating the map of layers on the $0$-vector space \[D_n(H_n(X))(0) \rightarrow D_\mathsf{fake}^n(H_n(X))(0)\]

    We will prove that this is an equivalence when $X=\Sigma^\infty_+ O(n)$ and use Theorem \ref{thm:norm for g spectra} to deduce that it coincides with the norm. This also solves the question for a general functor (with level-finite derivatives) by Proposition \ref{prop: Di equivalence and fake layers}. For any $V$ we have a diagram 
\[\begin{tikzcd}
	{D_n(H_n(\Sigma^\infty_+O(n)))(V)} && {D^n_\mathsf{fake}(H_n(\Sigma^\infty_+O(n)))(V)} \\
	{P_n(H_n(\Sigma^\infty_+O(n)))(V)} && {P^n_\mathsf{fake}(H_n(\Sigma^\infty_+O(n)))(V)} \\
	{\nat^h(R_V, H_n(\Sigma^\infty_+O(n)))} \\
	{\mathsf{RMod}^h_\mathsf{OEpi}(\mathsf{Free}_\mathsf{OEpi}(\Sigma^\infty_+O(n)), S^{\ast(V)})} && {\mathsf{RMod}^h_\mathsf{OEpi}(\mathsf{Free}_\mathsf{OEpi}(\Sigma^\infty_+O(n))^{\leq n}, (S^{\ast(V)})^{\leq n})}
	\arrow["{\circled{1}}", from=1-1, to=1-3]
	\arrow["{\circled{2}}"', from=1-1, to=2-1]
	\arrow["{\circled{2}}", from=1-3, to=2-3]
	\arrow["{\circled{3}}", from=2-1, to=2-3]
	\arrow["{\circled{4}}"', from=2-1, to=3-1]
	\arrow["{\circled{5}}"', from=3-1, to=4-1]
	\arrow["{\circled{6}}"', from=4-1, to=4-3]
	\arrow["{\circled{7}}"', from=4-3, to=2-3]
\end{tikzcd}\]
in which we want to show that the map labeled \circled{1} is an equivalence. The maps labelled \circled{2} are equivalence since $H_n(\Sigma^\infty_+O(n))$ is homogeneous, hence it suffices to show that the map labelled \circled{3} is an equivalence, which we do by showing that each map (\circled{4}-\circled{7}) in the factorization of the map labelled \circled{3} is an equivalence. The map labelled \circled{4} is an equivalence by the derived Yoneda Lemma, the map labelled \circled{5} is an equivalence by he Yoneda Lemma, since the Koszul dual derivatives are defined by taking natural transformations into a subcategory containing $H_n(\Sigma^\infty_+ O(n))$, the map labelled \circled{6} is an equivalence by the Yoneda lemma for the subcategory of $\mathsf{Vect}_\mathbb{R}$ of vector spaces of dimension $\leq n$, and finally, the map labelled \circled{7} is an equality, by definition. 
\end{proof}

\begin{ex}
The fake Weiss tower of representable functors agrees with the Weiss tower, i.e., for each $n$, the map
\[
P_n(R_V) \longrightarrow P^n_\mathsf{fake}(R_V),
\]
is an equivalence.
\end{ex}
\begin{proof}
By induction, it suffices to show that the result holds for the layers of the respective towers. By Theorem~\ref{thm: norm map finite derivatives}, the induced map on the layers of the towers may be identified with the dual norm map
\[
(\partial_n(R_V) \wedge S^{nV})_{hO(n)} \rightarrow (D_{O(n)}^\vee \wedge \partial_n(R_V) \wedge  S^{nV})^{hO(n)},
\]
hence it suffices to show that $\partial_n(R_V) \wedge S^{nV}$ is a finite Borel $O(n)$-spectrum. To see this, note that by Example~\ref{ex: Weiss derivative of rep}, the spectrum $\partial_n(R_V)$ is equivariantly equivalent to the spectrum
\[
\Omega^{nV}(\mathsf{Vect}_\bb{R}(\bb{R}^n, V) \wedge D_{O(n)}),
\]
which we already observed was free, so we can apply Theorem \ref{thm:norm for g spectra}. 
\end{proof}

\subsection{A Koszul duality model for Weiss derivatives} \label{subsection: Koszul duality model}

In light of the previous section, we seek to produce a functorial model of the Weiss derivatives which admits the structure of a right module over $K(\mathsf{OEpi})$. The issue with this is that our theory of Koszul dual derivatives was only ``correct'' when $\partial_i F$ was a finite spectrum for all $i$. This problem could be solved with ``pro-right modules'' as in \cite{ACOperads}, but there is a more direct approach that can be taken.

Let $c$ be a fixed simplicial cofibrant replacement of $\mathsf{RMod}_{K(\mathsf{OEpi})}$. Let \[\partial^K_\ast R_V\coloneq c(K(\partial^\ast R_V)) 
\cong c(K(S^{\ast V})).\]

\begin{definition}\label{definition: koszul model of derivatives}
    The Koszul model of the Weiss derivatives of a cofibrant $F:\mathsf{Vect}_\bb{R} \rightarrow \mathsf{Spec}$ is the right $K(\mathsf{OEpi})$-module
    \[\partial_\ast^K(F)\coloneq\int^{V \in \mathsf{Vect}_\bb{R}} \partial^K_\ast R_V \wedge F(V)\]
   
\end{definition}

Of course, a version of the Yoneda lemma asserts that in the case of $R_V$, these two definitions coincide. We can define a dual transform
\[
\Theta^K : \mathsf{RMod}_{K(\sf{OEpi})} \longrightarrow \Fun(\mathsf{Vect}_\bb{R}, \s),\]\[~S \longmapsto \mathsf{RMod}_{K(\sf{OEpi})}(\partial_\ast^K(R_{(-)}), S),
\]
which one should compare with Proposition~\ref{prop: Koszul adjunction} and Proposition~\ref{prop: koszul dual derivatives adjunction}.

\begin{prop}\label{prop: koszul derivatives adjunction}
There is a simplicial Quillen adjunction
    \[
\adjunction{\partial^K_\ast}{\mathsf{Fun}(\mathsf{Vect}_\bb{R},\s)}{\mathsf{RMod}_{K(\mathsf{OEpi})}}{\Theta^K}.
\]
\end{prop}
\begin{proof}
The adjunction exists by the standard calculus of (co)end arguments. It is Quillen since the right adjoint preserves (acyclic) fibrations by construction. 
\end{proof}

\begin{rem}
Since the categories above are stable, this adjunction lifts to an adjunction between stable $\infty$-categories and hence to spectral $\infty$-categories, as detailed in~Subsection~\ref{subsection: spectral homotopy}.
\end{rem}

\begin{ex}
    The work of the previous section showed that when $X$ is a nonequivariantly finite $O(n)$-spectrum 
    \[\partial^K_\ast (H_n(X)) \simeq K(\mathsf{Free}_{\mathsf{OEpi}}(X^\vee \wedge D_{O(n)}^\vee)) \simeq X.\]

    Every $O(n)$-spectrum is a colimit of finite $O(n)$-spectra, and so we conclude for an arbitrary $O(n)$-spectrum $Z$
    \[\partial^K_\ast (H_n(Z)) \simeq Z.\]
\end{ex}

The theory of Koszul derivatives leads to a satisfying calculation of the Koszul model for the Weiss derivatives of representable functors.

\begin{ex}\label{ex: Koszul derivatives reps}
    By Example \ref{ex: koszul dual derivative of reps}, we know $\partial^\ast R_V \simeq S^{\ast V}$. In particular, we can identify $\partial^i R_V$ with the stabilization of the one point compactification of $\mathbb{R}^i \otimes V$ where the action of the unstable $\mathsf{Vect}_\bb{R}$ is by tensoring in the left variable. We assert that unstably this action defines a cofibrant right module and matrix arithmetic shows that the decomposable elements of this right module are precisely the linear fat diagonal $\sf{DI}_\ast(V)$. This is because if we identify $\mathbb{R}^i \otimes V$ with $i$-tuples of elements in $V$, tensoring by a linear map $\mathbb{R}^i \rightarrow \mathbb{R}^j$ corresponds to replacing the $i$-tuple by a $j$-tuple of linear combinations of the original $i$-tuple. When $i>1$, the result is linearly dependent.

    Elementary model category arguments show that
    \[\mathsf{TAQ}(\Sigma^\infty S^{\ast V}) \simeq \Sigma^\infty S^{\ast V}/\sf{DI}_\ast(V)\]

    Thus, $\partial^K_\ast R_V \simeq  \Sigma^\infty S^{\ast V}/\sf{DI}_\ast(V)^\vee$ which by Atiyah duality recovers Example \ref{ex: Weiss derivative of rep}.
\end{ex}

\begin{thm}\label{thm:koszul derivatives are derivatives}
Let $F: \mathsf{Vect}_\bb{R} \to \s$ be a cofibrant functor. The right $K(\mathsf{OEpi})$-module $\partial^K_\ast F$ is cofibrant and there is a zigzag of equivalences of orthogonal sequences
    \[
    \partial^K_\ast F \simeq \partial_\ast F.
    \]
If $F$ is a finite cell complex, then there is a zigzag of equivalences of right $K(\mathsf{OEpi})$-modules
    \[
    \partial^K_\ast F \simeq K (\partial^\ast F).\]
\end{thm}
\begin{proof}
Cofibrancy holds since $\partial_\ast^K$ is a left Quillen functor by Proposition \ref{prop: koszul derivatives adjunction}.

We check that the Koszul dual and Spanier--Whitehead models for the Weiss derivatives agree. This implies the result by Lemma~\ref{lem: SW and OG agree}. By Example~\ref{ex: Koszul derivatives reps} and Example~\ref{ex: derivative of rep} the Koszul dual and Spanier--Whitehead models agree on representable functors. Since $F$ is cofibrant and the Koszul dual and Spanier--Whitehead models for the Weiss derivatives of $F$ are given by (homotopy) left Kan extending, the result follows. 

The second equivalence follows from a similar examination of representables (cf. Theorem \ref{thm: norm map finite derivatives}).
\end{proof}

\subsection{The fake Weiss tower} \label{subsection:enriched model of fake Weiss tower}

We revisit the fake orthogonal tower from the point of view of the Koszul model for the Weiss derivatives. This construction has the advantage that it works without the assumption that our functors have level-finite derivatives. 

 \begin{definition}
    Let $F: \mathsf{Vect}_\bb{R} \to \s$ be cofibrant. The \emph{fake Weiss approximation of $F$}
    \[F \longrightarrow P_\infty^\mathsf{fake}(F)\]
    is the unit of the $(\partial^K_\ast,\Theta^K)$-adjunction.
\end{definition}

Explicitly, the map
\[F(V) \longrightarrow P_\infty^\mathsf{fake}(F)(A)\]
is given by
\[F(V) \cong \nat(R_V,F) \xrightarrow{\partial^K_\ast} \mathsf{RMod}_{K(\mathsf{OEpi})}(\partial^K_\ast R_V,\partial^K_\ast F)\cong \mathsf{RMod}_{K(\mathsf{OEpi})}( K(S^{\ast V}),\partial^K_\ast F).\]
The fact that these mapping spectra are actually derived mapping spectra relies on the cofibrancy claim of Theorem \ref{thm:koszul derivatives are derivatives}.

\begin{definition}\label{def:fake weiss }
The \emph{fake Weiss tower} of a functor $F: \mathsf{Vect}_\bb{R} \to \s$ evaluated at $V$:
\[\begin{tikzcd}
	&& {F(V)} \\
	{P_\infty^\mathsf{fake} F(V)} & \cdots & {P_n^\mathsf{fake}F(V)} & \cdots & {P_1^\mathsf{fake}F(V)} & {P_0^\mathsf{fake}F(V)}
	\arrow[curve={height=6pt}, from=1-3, to=2-1]
	\arrow[from=1-3, to=2-3]
	\arrow[curve={height=-6pt}, from=1-3, to=2-5]
	\arrow[curve={height=-6pt}, from=1-3, to=2-6]
	\arrow[from=2-1, to=2-2]
	\arrow[from=2-2, to=2-3]
	\arrow[from=2-3, to=2-4]
	\arrow[from=2-4, to=2-5]
	\arrow[from=2-5, to=2-6]
\end{tikzcd}\]
    is the truncation tower for \[\mathsf{RMod}^h_{K(\mathsf{OEpi})}(\partial^K_\ast R_V,\partial^K_\ast F).\]
\end{definition}

\begin{prop}\label{prop: two fake towers agree}
    Definition \ref{def:fake weiss } and Definition \ref{def: fake weiss finite derivative} agree up to homotopy for cofibrant functors with level-finite derivatives.
\end{prop}

\begin{proof}
    Theorem \ref{thm:koszul duality for right modules} implies that Koszul duality yields an equivalence on derived mapping spectra for level-finite right modules. By Theorem \ref{thm:koszul derivatives are derivatives}, the Koszul derivatives and the Koszul dual derivatives are Koszul dual for functors with level-finite derivatives, so the result follows.
\end{proof}

We now reproduce the main result of Section \ref{subsection: fake weiss tower finite} without any restriction to functors with finite derivatives. We omit the supplementary results and proofs, as they consist of repeatedly appealing to Proposition \ref{prop: two fake towers agree} to extend the results of Section \ref{subsection: fake weiss tower finite} by colimits.

\begin{thm}\label{thm: norm map derivatives}
    For cofibrant $F$ the map $D_n(F)(V) \rightarrow D_n^\mathsf{fake}(F)(V)$
    is equivalent to the dual norm
    \[(\partial_n(F) \wedge S^{nV})_{hO(n)} \longrightarrow (\partial_n(F) \wedge D_{O(n)}^\vee \wedge S^{nV})^{hO(n)} .\]
\end{thm}

\section{Weiss towers and comonadic descent}
In this section, we show that the derivatives induce a comonad on the category of right modules in such a way that the Weiss tower of a functor $F: \mathsf{Vect}_\bb{R} \to \s$ is completely recovered from the structure of the derivatives $\partial_\ast F$ as a coalgebra over this comonad. In fact, we show that this classification holds both on the level of right $K(\mathsf{OEpi})$-modules and orthogonal sequences, so that the coalgebraic data on both these module categories has the same homotopy theory. 

For the rest of this section, fix once and for all a model for the Weiss derivatives as a functor  
\[
\partial_\ast : \Fun(\mathsf{Vect}_\bb{R}, \s) \longrightarrow \mathsf{RMod}_{\mc{C}},
\]
such that $\partial_\ast(R_{V})$ is a cofibrant right $\mc{C}$-module for each $V \in \mathsf{Vect}_\bb{R}$, where $\mc{C}$ denotes either $K(\mathsf{OEpi})$ or $O(\ast)$, specific constructions are given in Definition \ref{definition: spanier whitehead derivative} and Definition \ref{definition: koszul model of derivatives}. The bulk of this section is written $\infty$-categorically, and these model categorical requirements only exist to construct the $\infty$-categorical adjunction which gives rise to our comonad.

\subsection{The comonad} We begin by constructing the comonad at the heart of our descent story.

\begin{lem}\label{lem: derivative a left adjoint}
There is an adjoint pair
\[
\adjunction{\partial_\ast}{\Fun(\mathsf{Vect}_\bb{R}, \s)}{\sf{RMod}_{\mc{C}}}{\Phi},
\]
which is a simplicial Quillen adjunction.
\end{lem}
\begin{proof}
This follows from Lemma~\ref{lem: SW derivative a left adjoint} if $\mc{C}=O(\ast)$ and Proposition~\ref{prop: koszul derivatives adjunction} if $\mc{C}=K(\mathsf{OEpi})$.
\end{proof}

The simplicial Quillen adjunction induces an adjunction of spectral $\infty$-categories since the model structures in question are stable, which for ease of notation we continue to denote by
\[
\adjunction{\partial_\ast}{\Fun(\mathsf{Vect}_\bb{R}, \s)}{\sf{RMod}_{\mc{C}}}{\Phi},
\]
rather than cluttering the notation with $(-)_\infty$. In either case, this adjunction defines a comonad $\partial_\ast\Phi$ on the category of right $\mc{C}$-modules. The unit 
\[
\eta: \mathds{1}  \to \Phi\partial_\ast,
\]
provides $\partial_\ast F$ with the structure of a \emph{coalgebra} over the comonad $\partial_\ast\Phi$, i.e., there is a structure map
\[
\partial_\ast F \xrightarrow{\partial_\ast(\eta)} (\partial_\ast\Phi)\partial_\ast  F,
\]
in which the triangle identities provided the required associativity and unitality conditions. 

Recall from Subsection~\ref{subsection: filtration of right modules} that a right $\mc{C}$-module $R$ is said to be \emph{$n$-truncated} if $R_k = \ast$ for $k > n$. We now consider truncations of comonads on the category of right $\mc{C}$-modules. The right adjoint of Lemma~\ref{lem: derivative a left adjoint} plays well with respect to $n$-truncated right $\mc{C}$-module.

\begin{lem}\label{lem: phi of n-truncated}
If $R$ is an $n$-truncated right $\mc{C}$-module, then $\Phi(R)$ is $n$-polynomial.
\end{lem}
\begin{proof}
The proof follows verbatim from~\cite[Lemma 3.11]{ACClassification}.
\end{proof}

Note that the left adjoint sends $n$-polynomial functors to $n$-truncated right $\mc{C}$-modules since the $k$-th derivatives of an $n$-polynomial functor are trivial for $k > n$.

\subsection{$\partial_\ast$-completion}
We now introduce the notion of $\partial_\ast$-completion. Informally, this is a way of attempting to recover a functor from its derivatives together with its $\partial_\ast\Phi$-coalgebra structure.   

Given any $\partial_\ast\Phi$-coalgebra $A$, we define the \emph{comonadic cobar construction} $\mathsf{cobar}(\Phi, \partial_\ast\Phi, A)$ of $A$ to the totalization of the cosimplicial object 
\[\begin{tikzcd}
	{\Phi(A)} & {\Phi(\partial_\ast\Phi)(A)} & {\Phi(\partial_\ast\Phi)^2(A)} & \cdots
	\arrow[shift right=2, from=1-1, to=1-2]
	\arrow[shift left=2, from=1-1, to=1-2]
	\arrow[from=1-2, to=1-1]
	\arrow[shift right=2, from=1-2, to=1-3]
	\arrow[from=1-2, to=1-3]
	\arrow[shift left=2, from=1-2, to=1-3]
	\arrow[shift right, from=1-3, to=1-2]
	\arrow[shift left, from=1-3, to=1-2]
	\arrow[shift left, from=1-3, to=1-4]
	\arrow[shift left=3, from=1-3, to=1-4]
	\arrow[shift right, from=1-3, to=1-4]
	\arrow[shift right=3, from=1-3, to=1-4]
	\arrow[from=1-4, to=1-3]
	\arrow[shift left=2, from=1-4, to=1-3]
	\arrow[shift right=2, from=1-4, to=1-3]
\end{tikzcd}\]
with coface maps induced by the unit of the $(\partial_\ast, \Phi)$-adjunction and the coalgebra structure of $A$, and with codegeneracies induced by the counit of the $(\partial_\ast, \Phi)$-adjunction. Given a functor $F: \mathsf{Vect}_\bb{R} \to \s$, the unit of the adjunction applied to $F$ provide a map
\[
F \longrightarrow \Phi\partial_\ast(F),
\]
producing an augmentation
\[
F \longrightarrow \mathsf{cobar}(\Phi, \partial_\ast\Phi, \partial_\ast F),
\]
of the comonadic cobar construction applied to the coalgebra $\partial_\ast F$.

\begin{definition}
A functor $F: \mathsf{Vect}_\bb{R} \to \s$ is \emph{$\partial_\ast$-complete} if the map
\[
F \longrightarrow \sf{cobar}(\Phi, \partial_\ast\Phi, \partial_\ast F) \coloneq \sf{Tot}(\Phi(\partial_\ast\Phi)^\bullet (\partial_\ast F)),
\]
is an equivalence. 
\end{definition}

The following result is the Weiss calculus version of a fundamental theorem of Arone and Ching~\cite[Theorem 3.13]{ACClassification} in the setting of Goodwillie calculus, the proof of which is analogous, and we only sketch. Importantly, the following result immediately implies that convergent functors are $\partial_\ast$-complete. 

\begin{prop}\label{prop: completion and tower limit}
For a functor $F: \mathsf{Vect}_\bb{R} \to \s$, the $\partial_\ast$-completion map
\[
\eta: F \longrightarrow \sf{cobar}(\Phi, \partial_\ast\Phi, \partial_\ast F),
\]
is a retract of the map
\[
p_\infty: F \longrightarrow \underset{n}{\holim}~P_nF, 
\]
associated to the Weiss tower of $F$. 
\end{prop}
\begin{proof}
The fiber sequence
\[
D_kF \longrightarrow P_kF \longrightarrow P_{k-1}F, 
\]
provides a fiber sequence
\[
P_kF \longrightarrow P_{k-1}F \longrightarrow R_kF, 
\]
as $\s$ is a stable category. This fiber sequence induces a commutative diagram 
\[\begin{tikzcd}
	{P_kF} & {\sf{Tot}(P_k(\Phi(\partial_\ast\Phi)^\bullet\partial_\ast F))} \\
	{P_{k-1}F} & {\sf{Tot}(P_{k-1}(\Phi(\partial_\ast\Phi)^\bullet\partial_\ast F))} \\
	{R_kF} & {\sf{Tot}(R_k(\Phi(\partial_\ast\Phi)^\bullet\partial_\ast F))}
	\arrow[from=1-1, to=2-1]
	\arrow[from=2-1, to=3-1]
	\arrow[from=3-1, to=3-2]
	\arrow[from=2-1, to=2-2]
	\arrow[from=1-1, to=1-2]
	\arrow[from=1-2, to=2-2]
	\arrow[from=2-2, to=3-2]
\end{tikzcd}\]
in which the columns are fiber sequences. We employ induction to show that each horizontal arrow is a weak equivalence, which in turn implies it suffices to show that for eack $k$, the map
\[
R_kF \longrightarrow \sf{Tot}(R_k(\Phi(\partial_\ast\Phi)^\bullet \partial_\ast F)),
\]
is a weak equivalence. This is enough to conclude the proof since there is a commutative diagram
\[\begin{tikzcd}
	F & {\sf{Tot}(\Phi(\partial_\ast\Phi)^\bullet \partial_\ast F)} & {\underset{n}{\holim}~\sf{Tot}(\Phi((\partial_\ast\Phi)^\bullet \partial_\ast F)^{\leq n})} \\
	{\underset{n}{\holim}~P_nF} & {\underset{n}{\holim}~\sf{Tot}(P_n(\Phi(\partial_\ast\Phi)^\bullet \partial_\ast F))} & {\underset{n}{\holim}~\sf{Tot}(P_n(\Phi((\partial_\ast\Phi)^\bullet \partial_\ast F)^{\leq n})).}
	\arrow["{p_\infty}"', from=1-1, to=2-1]
	\arrow["\eta", from=1-1, to=1-2]
	\arrow["\simeq"', from=2-1, to=2-2]
	\arrow["\simeq", from=1-2, to=1-3]
	\arrow[from=1-2, to=2-2]
	\arrow[from=2-2, to=2-3]
	\arrow["\simeq", from=1-3, to=2-3]
\end{tikzcd}\]
To show that the map 
\[
R_kF \longrightarrow \sf{Tot}(R_k(\Phi(\partial_\ast\Phi)^\bullet \partial_\ast F)),
\]
is a weak equivalence. Define a functor
\[
\Psi_k: \sf{RMod}_{\mc{C}} \longrightarrow \Fun(\mathsf{Vect}_\bb{R}, \s),\ A \longmapsto (V \mapsto (S^{\bb{R}^k \otimes V}  \wedge A_k)_{hO(k)}),
\]
then $D_kF \simeq \Psi_k(\partial_\ast F)$, so there is a natural equivalence $R_kF \simeq R_k\Psi_k\partial_\ast F$, through which the map we are interested in is equivalent to the map 
\[
R_k\Psi_k(\partial_\ast F) \longrightarrow \sf{Tot}(R_k\Psi_k((\partial_\ast\Phi)^{\bullet+1}\partial_\ast F)),
\]
which is the coaugmentation map from a cosimplicial object with extra codegeneracies, hence a weak equivalence in $\Fun(\mathsf{Vect}_\bb{R}, \s)$. 
\end{proof}

\begin{cor}\label{cor: complete if converge}
If $F: \mathsf{Vect}_\bb{R} \to \s$ is a functor whose Weiss tower converges at $V\in \mathsf{Vect}_\bb{R}$, then $F$ is $\partial_\ast$-complete at $V$, i.e., the map
\[
\eta_V: F(V) \longrightarrow \sf{cobar}(\Phi, \partial_\ast\Phi, \partial_\ast F)(V),
\]
is a weak equivalence of spectra.
\end{cor}
\begin{proof}
Recall the commutative diagram
\[\begin{tikzcd}
	F & {\sf{Tot}(\Phi(\partial_\ast\Phi)^\bullet \partial_\ast F)} & {\underset{n}{\holim}~\sf{Tot}(\Phi((\partial_\ast\Phi)^\bullet \partial_\ast F)^{\leq n})} \\
	{\underset{n}{\holim}~P_nF} & {\underset{n}{\holim}~\sf{Tot}(P_n(\Phi(\partial_\ast\Phi)^\bullet \partial_\ast F))} & {\underset{n}{\holim}~\sf{Tot}(P_n(\Phi((\partial_\ast\Phi)^\bullet \partial_\ast F)^{\leq n})).}
	\arrow["{p_\infty}"', from=1-1, to=2-1]
	\arrow["\eta", from=1-1, to=1-2]
	\arrow["\simeq"', from=2-1, to=2-2]
	\arrow["\simeq", from=1-2, to=1-3]
	\arrow[from=1-2, to=2-2]
	\arrow[from=2-2, to=2-3]
	\arrow["\simeq", from=1-3, to=2-3]
\end{tikzcd}\]
If the left-hand vertical map is an equivalence, then one can argue along the diagram using the two-out-of-three property for equivalences to conclude that the map $\eta$ is an equivalence.
\end{proof}

\subsection{Spectral $\infty$-categories and comonadicity}

In the next subsection, we will prove that the $(\partial_\ast, \Phi)$-adjunction is comonadic on the level of spectral $\infty$-categories. The passage to the associated $\infty$-category is mostly technical: coalgebras over comonads do not interact well with model category theory. This approach is an alternative to using the replacement of adjunctions approach developed by Ching and Riehl~\cite{ChingRiehl}, and used heavily by Arone and Ching~\cite{ACClassification}. In many ways, this approach absorbs the coherences, making it conceptually much easier.

We now state the enriched version of the Barr-Beck-Lurie comonadicity theorem originally due to Heine~\cite[Theorem 1.8]{Heine}. The cited result is about algebras, but by dualizing using that $\mc{V}$-enriched comonads are $\mc{V}$-enriched monads in the opposite $\mc{V}$-category and, that the $\mc{V}$-category of coalgebras for a $\mc{V}$-enriched comonad is the opposite $\mc{V}$-category of the $\mc{V}$-category of algebras of the corresponding V-enriched monad, i.e., $\sf{CoAlg}(\mc{M}) = (\sf{Alg}(\mc{M}^\op))^\op$. We state the theorem only in the generality we require. 

\begin{thm}[{\cite[Theorem 1.8]{Heine}}]\label{thm: comonadicity}
A spectrally enriched functor $F: \mc{M} \to \mc{N}$ between spectral $\infty$-categories is comonadic if and only if
\begin{enumerate}
    \item the spectrally enriched functor $F: \mc{M} \to \mc{N}$ has an enriched right adjoint $G: \mc{N} \to \mc{M}$;
    \item The underlying functor $F: \mc{M} \to \mc{N}$ on the level of $\infty$-categories is conservative; and,
    \item every $F$-split cosimplicial object in $\mc{M}$, admits an enriched limit that is preserved by the underlying functor $F$.
\end{enumerate}
\end{thm}

We now record the following lemma which reduces Theorem~\ref{thm: comonadicity}$(3)$ to the case of a canonical cosimplicial object, the proof of which is due to Heuts~\cite{HeutsGoodwillieApprox} in the non-enriched $\infty$-categorical setting (see also~\cite[Proposition 6.1.4]{Peroux}) and readily extended to the enriched setting.

\begin{lem}\label{lem: FG resolution}
Let $F: \mc{M} \to \mc{N}$ be a spectrally enriched functor left adjoint to $G: \mc{N} \to \mc{M}$. Every $F$-split cosimplicial object in $\mc{M}$ admits an enriched limit that is preserved by the underlying functor $F$ if and only if, for the canonical $GF$-resolution
\[\begin{tikzcd}
	X & {GF(X)} & {GFGF(X)} & \cdots
	\arrow[from=1-1, to=1-2]
	\arrow[shift right=2, from=1-2, to=1-3]
	\arrow[shift left=2, from=1-2, to=1-3]
	\arrow[from=1-3, to=1-2]
	\arrow[shift right=2, from=1-3, to=1-4]
	\arrow[from=1-3, to=1-4]
	\arrow[shift left=2, from=1-3, to=1-4]
	\arrow[shift left, from=1-4, to=1-3]
	\arrow[shift right, from=1-4, to=1-3]
\end{tikzcd}\]
the induced map
\[
X \longrightarrow \mathsf{Tot}(G(FG)^{\bullet}F(X))=\mathsf{cobar}(G, FG, F)(X),
\]
is an equivalence for all $X \in \mc{M}$.
\end{lem}

\subsection{Comonadic descent}
Proposition~\ref{prop: completion and tower limit} tells us that $n$-polynomial functors can be recovered from the $n$-truncated derivatives, i.e., from the first $n$-many derivatives. To ease notation, we will denote the $n$-truncated derivatives by $\partial_{\leq n}F$.

\begin{lem}\label{lem: polynomials are complete}
Let $n$ be a non-negative integer. For any functor $F: \mathsf{Vect}_\bb{R} \to \s$, the map 
\[
P_nF \longrightarrow \sf{cobar}(\Phi, \partial_\ast\Phi, \partial_{\leq n}F),
\]
is a levelwise weak equivalence. Moreover, the canonical map $P_nF \to P_{n-1}F$ is equivalent to the map
\[
\sf{cobar}(\Phi, \partial_\ast\Phi, \partial_{\leq n}F) \longrightarrow \sf{cobar}(\Phi, \partial_\ast\Phi, \partial_{\leq n-1}F)
\]
induced by the truncation tower of $\partial_\ast(F)$ as a right $\mc{C}$-module. 
\end{lem}
\begin{proof}
By Corollary~\ref{cor: complete if converge}, we have that the map 
\[
P_nF \longrightarrow \sf{cobar}(\Phi, \partial_\ast\Phi, \partial_{\ast}(P_nF)),
\]
is a levelwise weak equivalence. It hence suffices to prove that the map of cosimplicial objects
\[
\Phi (\partial_\ast\Phi)^\bullet \partial_{\leq n}F \longrightarrow \Phi(\partial_\ast\Phi)^\bullet\partial_\ast(P_nF),
\]
is a weak equivalence. This last follows readily from the fact that the $m$-th derivative of $F$ and the $m$-th derivative of $P_nF$ agree for $m \leq n$. 
\end{proof}

We introduce some terminology.

\begin{definition}
Let $n$ be a non-negative integer and let $T$ be a comonad on $\sf{RMod}_{\mc{C}}$. Naturality of truncation allows us to define a comonad $T^{\leq n}$ on the subcategory of $n$-truncated right $\mc{C}$-modules by
\[
T^{\leq n}(A) = (TA)^{\leq n}.
\]
We say that a coalgebra over the comonad $T^{\leq n}$ is an \emph{$n$-truncated $T$-coalgebra}.
\end{definition}

An important example is given as follows.

\begin{ex}
Let $T$ be a comonad on the category of right $\mc{C}$-modules. If $A$ is a $T$-coalgebra, then the $n$-truncation $A^{\leq n}$ of $A$ is an $n$-truncated $T$-coalgebra.
\end{ex}

We now wish to prove the analogue of \cite[Theorem 3.19]{ACClassification} which exhibits an equivalence between the homotopy theory of $n$-polynomial functors  and $n$-truncated $(\partial_\ast\Phi)$-coalgebras in $\sf{RMod}_\mc{O}$. To do this, we invoke Theorem~\ref{thm: comonadicity}, to show that the $(\partial_\ast, \Phi)$-adjunction is comonadic, and hence our proof differs somewhat from the proof of Arone and Ching.

\begin{thm}\label{thm: derivative comonadic}
Let $n$ be a non-negative integer. The adjoint pair
\[
\adjunction{\partial_\ast}{\poly{n}(\mathsf{Vect}_\bb{R}, \s)}{\sf{RMod}_{\mc{C}}}{\Phi},
\]
is comonadic, inducing an equivalence of $\infty$-categories
\[
\poly{n}(\mathsf{Vect}_\bb{R}, \s) \cong \sf{CoAlg}^{\leq n}_{\partial_\ast\Phi}(\sf{RMod}_\mathcal{C}),
\]
between $n$-polynomial functors and $n$-truncated $\partial_\ast\Phi$-coalgebras in orthogonal sequences.
\end{thm}
\begin{proof}
We verify the conditions $(1)$-$(3)$ of Theorem~\ref{thm: comonadicity}, to show that the functor
\[
\partial_\ast : \Fun(\mathsf{Vect}_\bb{R}, \s) \longrightarrow \mathsf{RMod}_{\mc{C}},
\]
is comonadic. By Lemma~\ref{lem: derivative a left adjoint}, there is a spectrally enriched adjunction (of $\infty$-categories)
\[
\adjunction{\partial_\ast}{\Fun(\mathsf{Vect}_\bb{R}, \s)}{\mathsf{RMod}_{\mc{C}}}{\Phi},
\]
coming from the spectral Quillen adjunction, which verifies condition $(1)$ of Theorem~\ref{thm: comonadicity}. For condition $(2)$ of Theorem~\ref{thm: comonadicity} it suffices to show that the derivatives functor is conservative on the level of underlying $\infty$-categories. Let $f: E \to F$ be a map of $n$-polynomial functors such that $\partial_\ast(f) : \partial_\ast E \to \partial_\ast F$ is an equivalence of orthogonal sequences. To show that $f$ is an equivalence of $n$-polynomial functors,  consider the following diagram
\[\begin{tikzcd}
	E & {P_nE} & {\mathsf{cobar}(\Phi, \partial_\ast\Phi, \partial_{\leq n}E)} \\
	F & {P_nF} & {\mathsf{cobar}(\Phi, \partial_\ast\Phi, \partial_{\leq n}F)}
	\arrow["f", from=1-1, to=2-1]
	\arrow[from=1-1, to=1-2]
	\arrow[from=1-2, to=2-2]
	\arrow[from=2-1, to=2-2]
	\arrow[from=1-2, to=1-3]
	\arrow[from=1-3, to=2-3]
	\arrow[from=2-2, to=2-3]
\end{tikzcd}\]
in which the left-most horizontal maps are equivalences since $E$ and $F$ are both $n$-polynomial, and the right-most horizontal maps are equivalences by Lemma~\ref{lem: polynomials are complete}. Hence, it suffices to show that the induced map
\[
\sf{cobar}(\Phi, \partial_\ast\Phi, \partial_{\leq n}E) \longrightarrow \sf{cobar}(\Phi, \partial_\ast\Phi, \partial_{\leq n}F),
\]
is an equivalence of orthogonal sequences, but this follows immediately from the equivalence 
\[
\partial_\ast(f) : \partial_\ast E \to \partial_\ast F,
\]
and the construction of the (cosimplicial) cobar complex. 

It is now left to verify $(3)$ of Theorem~\ref{thm: comonadicity}. By Lemma~\ref{lem: FG resolution} it suffices to show that for the canonical $\Phi\partial_\ast$-resolution
\[\begin{tikzcd}
	F & {\Phi\partial_\ast(F)} & {\Phi\partial_\ast\Phi\partial_\ast(F)} & \cdots
	\arrow[from=1-1, to=1-2]
	\arrow[shift right=2, from=1-3, to=1-4]
	\arrow[from=1-3, to=1-4]
	\arrow[shift left=2, from=1-3, to=1-4]
	\arrow[shift right, from=1-2, to=1-3]
	\arrow[shift left, from=1-2, to=1-3]
\end{tikzcd}\]
of an $n$-polynomial functor $F$, the induced map
\[
F \longrightarrow \mathsf{Tot}(\Phi(\partial_\ast\Phi)^{\bullet}\partial_\ast F),
\]
is an equivalence in $\poly{n}(\mathsf{Vect}_\bb{R}, \s)$. To see this, consider the commutative diagram
\[\begin{tikzcd}
	F & {\mathsf{Tot}(\Phi(\partial_\ast\Phi)^{\bullet}\partial_\ast F)} & {\mathsf{cobar}(\Phi, \partial_\ast\Phi, \partial_\ast F)} \\
	{P_nF} & {\mathsf{Tot}(\Phi(\partial_\ast\Phi)^{\bullet}\partial_\ast P_nF)} & {\mathsf{cobar}(\Phi, \partial_\ast\Phi, \partial_\ast P_nF)}
	\arrow[from=1-1, to=1-2]
	\arrow[Rightarrow, no head, from=1-2, to=1-3]
	\arrow[Rightarrow, no head, from=2-2, to=2-3]
	\arrow["\simeq", from=1-3, to=2-3]
	\arrow["\simeq", from=1-2, to=2-2]
	\arrow["\simeq"', from=1-1, to=2-1]
	\arrow[from=2-1, to=2-2]
\end{tikzcd}\]
in which all vertical maps are equivalences. It suffices to show that the lower-horizontal map is an equivalence, but this is precisely the content of Lemma~\ref{lem: polynomials are complete}.
\end{proof}

\subsection{Kuhn-McCarthy Classification}
We now prove the Weiss calculus version of the Kuhn-McCarthy Classification Theorem, see e.g.,~\cite{kuhn_2004, McCarthyDual}, which classifies the map $P_nF \to P_{n-1}F$ by a universal fibration using the coalgebra structure. Since the comonadic data associated to the various versions of the derivatives agrees up to homotopy, we assume now that the derivative functor takes values in orthogonal sequences so that the computations become simpler.

\begin{prop}\label{prop: McCarthy1}
Let $n \geq 2$. For every functor $F: \mathsf{Vect}_\bb{R} \to \s$, there is a homotopy pullback square
\[\begin{tikzcd}
	{P_nF} & {\Phi\partial_nF} \\
	{P_{n-1}F} & {P_{n-1}(\Phi\partial_nF)}
	\arrow[from=1-1, to=2-1]
	\arrow[from=1-1, to=1-2]
	\arrow[from=1-2, to=2-2]
	\arrow[from=2-1, to=2-2]
\end{tikzcd}\]
where $\Phi\partial_nF$ denotes the value of $\Phi$ on the orthogonal sequence consisting of the $n$-th derivative concentrated in degree $n$. 
\end{prop}
\begin{proof}
This is analogous to the content of \cite[Proposition 4.14]{ACClassification}, which we sketch here. By Lemma~\ref{lem: phi of n-truncated}, $\Phi\partial_nF$ is $n$-polynomial, hence it suffices to show that the map $P_nF \to \Phi\partial_nF$ is a $D_n$-equivalence, i.e., that the map on fibers is an equivalence. The map in question factors as
\[
D_n(P_nF) \longrightarrow D_n(\Phi\partial_\ast(P_nF)) \simeq D_n(\Phi\partial_{\leq n}F) \longrightarrow D_n(\Phi\partial_nF),
\]
in which the first map is an equivalence by the Weiss calculus version of ~\cite[Proposition 3.20]{ACClassification}, ultimately following from our computations of the fake tower in Theorem \ref{thm: norm map derivatives}, and the second map is an equivalence by Lemma~\ref{lem: phi of n-truncated}.
\end{proof}

In the following result, we calculate the right-hand side of this pullback square. Note that for orthogonal sequences, the right adjoint to $\partial_\ast$ may be described as
\[
\Phi(R)(V) = \prod_{n \in {O(\ast)}}\s(\partial_n(\Sigma^\infty R_V), R(n))^{O(n)},
\]
since there are no nontrivial non-automorphisms in $O(\ast)$. The statement we give below differs from Theorem~\ref{thm: intro pullback} from the introduction, only by applying the norm map to the top right corner of the diagram.

\begin{cor}\label{cor: McCarthy-Kuhn}
Let $n \geq 2$. For every functor $F: \mathsf{Vect}_\bb{R} \to \s$, there is a homotopy pullback square
\[\begin{tikzcd}
	{P_nF(V)} & {(S^{nV}/\sf{DI}_n(V) \wedge \partial_n F)_{hO(n)}} \\
	{P_{n-1}F(V)} & {(\Sigma \sf{DI}_n(V) \wedge \partial_n F)_{hO(n)}}
	\arrow[from=1-1, to=2-1]
	\arrow[from=1-1, to=1-2]
	\arrow[from=1-2, to=2-2]
	\arrow[from=2-1, to=2-2]
\end{tikzcd}\]	
natural in $V \in \mathsf{Vect}_\bb{R}$.
\end{cor}
\begin{proof}
The proof is a calculation of the right-hand side of the pullback square of Proposition~\ref{prop: McCarthy1}. First note, for a orthogonal sequence $A$, 
\begin{align*}
\Phi (A)(V) 
&= \prod_{n} \s(\partial_n(\Sigma^\infty R_V), A_n)^{O(n)} \\
&\simeq \prod_n \s((\s(D_{O(n)} \wedge \mathsf{Vect}_\bb{R}(\bb{R}^n, V), S^{nV})^\vee, A_n)^{O(n)} \\
&\simeq \prod_n \left(\s(D_{O(n)} \wedge \mathsf{Vect}_\bb{R}(\bb{R}^n, V), S^{nV}) \wedge A_n\right)^{O(n)} \\
&\simeq \prod_n \left(D_{O(n)}^\vee \wedge (S^{nV}/\sf{DI}_n(V)) \wedge A_n\right )^{O(n)} \\
&\simeq \prod_n \left(D_{O(n)}^\vee \wedge (S^{nV}/\sf{DI}_n(V)) \wedge A_n\right )^{hO(n)}.
\end{align*}

Hence, the upper-right term in the pullback square of Proposition~\ref{prop: McCarthy1} is precisely
\[
(D_{O(n)}^\vee \wedge(S^{nV}/\sf{DI}_n(V)) \wedge \partial_n F)^{hO(n)},
\]
when viewing the spectrum $\partial_nF$ as an orthogonal sequence concentrated in degree $n$. Since the spectrum $D_{O(n)}^\vee \wedge(S^{nV}/\sf{DI}_n(V)) \wedge \partial_n F$ is an $O(n)$-spectrum of the form $X \wedge Y$, where $X$ is an arbitrary $O(n)$-spectrum and $Y$ is a finite $O(n)$-cell spectrum, we know that the associated norm map
\[
\sf{Nm}_{O(n)} : \left( D_{O(n)} \wedge D_{O(n)}^\vee \wedge(S^{nV}/\sf{DI}_n(V)) \wedge \partial_n F \right)_{hO(n)} \xlongrightarrow{\simeq} \left(D_{O(n)}^\vee \wedge(S^{nV}/\sf{DI}_n(V)) \wedge \partial_n F \right)^{hO(n)}
\]
is an equivalence by induction over the cells of $Y$ and applying Klein's Theorem \ref{thm:norm for g spectra}. It follows that
\[
\Phi\partial_nF(V) \simeq (S^{nV}/\sf{DI}_n(V) \wedge \partial_n F)_{hO(n)}.
\]

It is left to calculate $P_{n-1}(\Phi \partial_nF)$. By the above reasoning it suffices to calculate the $(n-1)$-polynomial approximation of the source of the norm map, or equivalently of the functor 
\[
V \longmapsto (S^{nV}/\sf{DI}_n(V) \wedge \partial_n F)_{hO(n)}. 
\]
The existence of the (co)fiber sequence 
\[
(S^{nV} \wedge \partial_n F)_{hO(n)} \longrightarrow (S^{nV}/\sf{DI}_n(V) \wedge \partial_nF)_{hO(n)} \longrightarrow (\Sigma \sf{DI}_n(V) \wedge \partial_n F)_{hO(n)},
\]
implies that it suffices to show that the functor
\[
V \longmapsto (\Sigma \sf{DI}_n(V) \wedge \partial_n F)_{hO(n)},
\]
is $(n-1)$-polynomial, since the first term of the fiber sequence is $n$-homogeneous. This last follows as it is a homotopy colimit of $(n-1)$-polynomial functors, similar to the fat diagonal in Goodwillie calculus.
\end{proof}

\bibliography{references}
\bibliographystyle{alpha}
\end{document}